\documentclass[a4paper,12pt]{article}%{scrartcl}

\usepackage{amsmath,amsfonts,amssymb}
\usepackage[thmmarks,amsmath]{ntheorem}
\usepackage{enumerate}
\usepackage{a4wide}
\usepackage{graphics}
\usepackage{epsfig}
\usepackage{curves}
\usepackage{epic,eepic}
\usepackage{ntheorem}
\usepackage{authblk}
\usepackage{lscape}
\allowdisplaybreaks

\newtheorem{proposition}{Proposition}[section]
\newtheorem{theorem}[proposition]{Theorem}
\newtheorem{corollary}[proposition]{Corollary}

\newtheorem{definition}[proposition]{Definition}

\newtheorem{algorithm}[proposition]{Algorithm}
\newtheorem{example}[proposition]{Example}

\theorembodyfont{\normalfont}
\newenvironment{proof}{\noindent\textbf{Proof.}
  }{\hspace*{\fill} $\square$ \\[1em]}

\newcommand{\xbold}{{\bf x}}
\newcommand{\zbold}{{\bf z}}

\newcommand{\pbold}{{\bf p}}

\newcommand{\ubold}{{\bf u}}
\newcommand{\lbold}{{\bf l}}
\newcommand{\ybold}{{\bf y}}
\newcommand{\wbold}{{\bf w}}

\renewcommand{\Re}{{\mathbb R}}
\newcommand{\Z}{{\mathbb Z}}

\def\trt{^{\scriptscriptstyle T}}

\title{Computing all solutions of Nash equilibrium problems with discrete strategy sets}
\author{Simone Sagratella\thanks{The work of the author has been partially supported by Avvio alla Ricerca 2015 Sapienza University of Rome, under grant 488.}}
\affil{Department of Computer, Control and Management Engineering Antonio Ruberti, Sapienza University of Rome, Via Ariosto 25, 00185 Roma, Italy\\
sagratella@dis.uniroma1.it}

\begin{document}

\maketitle

\begin{abstract}
The Nash equilibrium problem is a widely used tool to model non-cooperative games. 
Many solution methods have been proposed in the literature to compute solutions of Nash equilibrium problems with continuous strategy sets, but, besides some specific methods for some particular applications, there are no general algorithms to compute solutions of Nash equilibrium problems in which the strategy set of each player is assumed to be discrete.
We define a branching method to compute the whole solution set of Nash equilibrium problems with discrete strategy sets. This method is equipped with a procedure that, by fixing variables, effectively prunes the branches of the search tree. Furthermore, we propose a preliminary procedure that by shrinking the feasible set improves the performances of the branching method when tackling a particular class of problems.
Moreover, we prove existence of equilibria and we propose an extremely fast Jacobi-type method which leads to one equilibrium for a new class of Nash equilibrium problems with discrete strategy sets. Our numerical results show that all proposed algorithms work very well in practice.
\end{abstract}

\section{Introduction}\label{sec: introduction}

The Nash equilibrium problem is a key model in game theory that has been widely used in many fields since fifties, see \cite{Nash50,Nash51}. Anyway, a strong interest in the numerical solution of realistic Nash games is a relatively recent research topic. Several algorithms have been proposed for the computation of one solution of Nash equilibrium problems, see e.g. monograph \cite{FacchPangBk} and the references therein, and of generalized Nash equilibrium problems, see e.g. \cite{FacchKanSu} and \cite{dreves2011solution,dreves2011nonsmooth,dreves2012nonsmooth,facchinei2011partial,facchlampsagr2011sur,facchinei2011computation,facchinei2014solving,NabTsengFuk09,FukPang09}.
All these methods assume that the feasible region of all players is continuous, and, besides some specific procedures for some particular applications and to the best of our knowledge, no numerical methods for the solution of any Nash equilibrium problem with discrete strategy spaces have been proposed so far. This constitutes an important gap in the literature, since, in general, it is not clear how it is possible to compute equilibria of many non-cooperative game models that are designed so that their variables represent indivisible quantities, see e.g. \cite{bikhchandani1997competitive,gabriel2013solving,gabriel2013solving2}.

Another important numerical topic is the computation of the whole solution set of Nash equilibrium problems. This issue becomes crucial when a selection of the equilibria is in order. Some methods were proposed for computing the whole equilibrium set of generalized Nash games with continuous strategy sets, see \cite{dreves2011nonsmooth,facchinei2011computation,NabTsengFuk09}, but, as said above, there are no general methods for discrete games.

In this work we present a branching method to compute all solutions of any Nash equilibrium problem with discrete strategy sets (Section \ref{sec: branching}). We define an effective fixing strategy which is useful in order to prune the branches of the search tree (Subsection \ref{subsec: proocedure F}). Moreover, we define an algorithmic framework for a class of Nash equilibrium problems with discrete strategy sets that, by using the branching method, efficiently yields the whole equilibrium set (Subsection \ref{subsec: algorithmic framework}). Furthermore, we define a new class of Nash equilibrium problems with discrete strategy sets for which: (i) we propose an extremely fast Jacobi-type algorithm to compute one equilibrium, (ii) we prove existence of equilibria, and (iii) we give an economic interpretation as a standard pricing game (Section \ref{sec:class}).

In Section \ref{sec: numerical}, we show that algorithms proposed in this paper work very well in practice. We deal with problems up to 1000 variables.

\section{Problem description and preliminary results}
\label{sec: preliminaries}

We consider a Nash equilibrium problem (NEP) with
 $N$ players and denote by
$ x^\nu \in \mathbb R^{n_\nu}$ the vector representing the
$\nu$-th player's strategy. 
We further define the vector
$\xbold^{-\nu} := (x^{\nu'})_{\nu\neq \nu' =1}^N$ and write $\mathbb R^{n}\ni\xbold :=
(x^\nu,\xbold^{-\nu})$, where $n := n_1 + \cdots + n_N$.

Each player has to solve an
optimization problem in which the objective function depends on
other players' variables, while the feasible set is defined by convex constraints depending on player's variables only,
plus integrality constraints:
\begin{equation} \label{eq: prob}
\begin{array}{l}
 \min_{x^\nu} \theta_\nu (x^\nu, \xbold^{-\nu}) \\[0.5em]
\quad g^\nu (x^\nu) \le 0 \\[0.5em]
\quad x^\nu \in \Z^{n_\nu},
\end{array} 
\end{equation}
where $\theta_\nu \, :\, \Re^{n} \, \rightarrow \, \Re$ and $g^\nu \, :\, \Re^{n_\nu} \, \rightarrow \, \Re^{m_\nu}$.
We call this problem as discrete NEP.

By relaxing integrality constraints in \eqref{eq: prob} we obtain a NEP in which each player $\nu$ solves the following optimization problem:
\begin{equation} \label{eq: continuous prob}
\begin{array}{l}
 \min_{x^\nu} \theta_\nu (x^\nu, \xbold^{-\nu}) \\[0.5em]
\quad g^\nu (x^\nu) \le 0,
\end{array} 
\end{equation}
we call this problem as continuous NEP.

Let us define
$$
X^\nu :=\{ x^\nu \in \Re^{n_\nu}: g^\nu (x^\nu) \leq 0\}, \qquad X := \prod_{\nu=1}^N X^\nu,
$$
a point $\xbold_*\in X \cap \Z^n$ is a solution (or an equilibrium) of the discrete
NEP if, for all $\nu$, $x_*^\nu$ is an optimal solution of problem (\ref{eq: prob})
when $\xbold^{-\nu}$ is fixed to $\xbold^{-\nu}_*$, that is:
$$
\theta_\nu (x^\nu_*, \xbold^{-\nu}_*) \leq \theta_\nu (x^\nu, \xbold^{-\nu}_*), \qquad \forall \, x^\nu \in X^\nu \cap \Z^{n_\nu}.
$$
On the other hand, a point $\bar \xbold \in X$ is a solution (or an equilibrium) of the continuous NEP if, for all $\nu$, $\bar x^\nu$ is an optimal solution of problem \eqref{eq: continuous prob}
when $\xbold^{-\nu}$ is fixed to $\bar \xbold^{-\nu}$, that is:
$$
\theta_\nu (\bar x^\nu, \bar \xbold^{-\nu}) \leq \theta_\nu (x^\nu, \bar \xbold^{-\nu}), \qquad \forall \, x^\nu \in X^\nu.
$$
For the sake of simplicity, in this paper we make the following blanket assumptions for each player $\nu$: $\theta_\nu$ is continuously differentiable and it is convex as a function of $x^\nu$ alone, and $g^\nu$ is continuously differentiable and convex.

Let $F$ be the operator comprised by the objective function gradients of all players:
$$
F^\nu (\xbold) := \nabla_{x^\nu} \theta_\nu (\xbold), \qquad
F(\xbold) := \left(\begin{array}{c}
            F^1 (\xbold) \\
            \vdots \\
            F^N (\xbold)
           \end{array} \right).
$$
It is well known that if $JF(\xbold)$ is symmetric for all $\xbold \in X$, then a function $f : \Re^n \to \Re$ exists such that $\nabla_\xbold f(\xbold) = F(\xbold)$ for all $\xbold \in X$, and then the set of all equilibria of the continuous NEP, defined by \eqref{eq: continuous prob}, coincides with the solution set of the following optimization problem:
\begin{equation*}
\begin{array}{l}
 \min_{\xbold} f (\xbold) \\[0.5em]
\quad \xbold \in X.
\end{array} 
\end{equation*}
This nice connection does not hold for discrete NEPs. It is very easy to give an example of a discrete NEP with $JF(\xbold)$ symmetric for all $\xbold \in X$ and for which the solution set of the corresponding discrete optimization problem, that is
\begin{equation}\label{eq: optimization problem}
\begin{array}{l}
 \min_{\xbold} f (\xbold) \\[0.5em]
\quad \xbold \in X \cap \Z^n,
\end{array} 
\end{equation}
does not contain all discrete equilibria.

\begin{example}\label{ex: one}
There are two
players each controlling one variable. Players' problems are
\begin{align*}
 \displaystyle\min_{x^1} \, \theta_1(x^1,x^2) & = \frac{9}{2} (x^1)^2 + 7 x^1 x^2 - 72 x^1 \\
& 0 \le x^1 \le 9 \\
& x^1 \in \Z, \\[1em]
 \displaystyle\min_{x^2} \, \theta_2(x^1,x^2) & = \frac{9}{2} (x^2)^2 + 7 x^1 x^2 - 72 x^2 \\
& 0 \le x^2 \le 9 \\
& x^2 \in \Z.
\end{align*}
This discrete NEP has the following equilibria: $(3,6)$, $(4,5)$, $(5,4)$ and $(6,3)$, see Figure \ref{fig: ex1}.
\begin{figure}[!ht]
 \begin{minipage}{0.49\textwidth}
 \centering
 \hspace*{-2em}
 \includegraphics[width=1\textwidth]{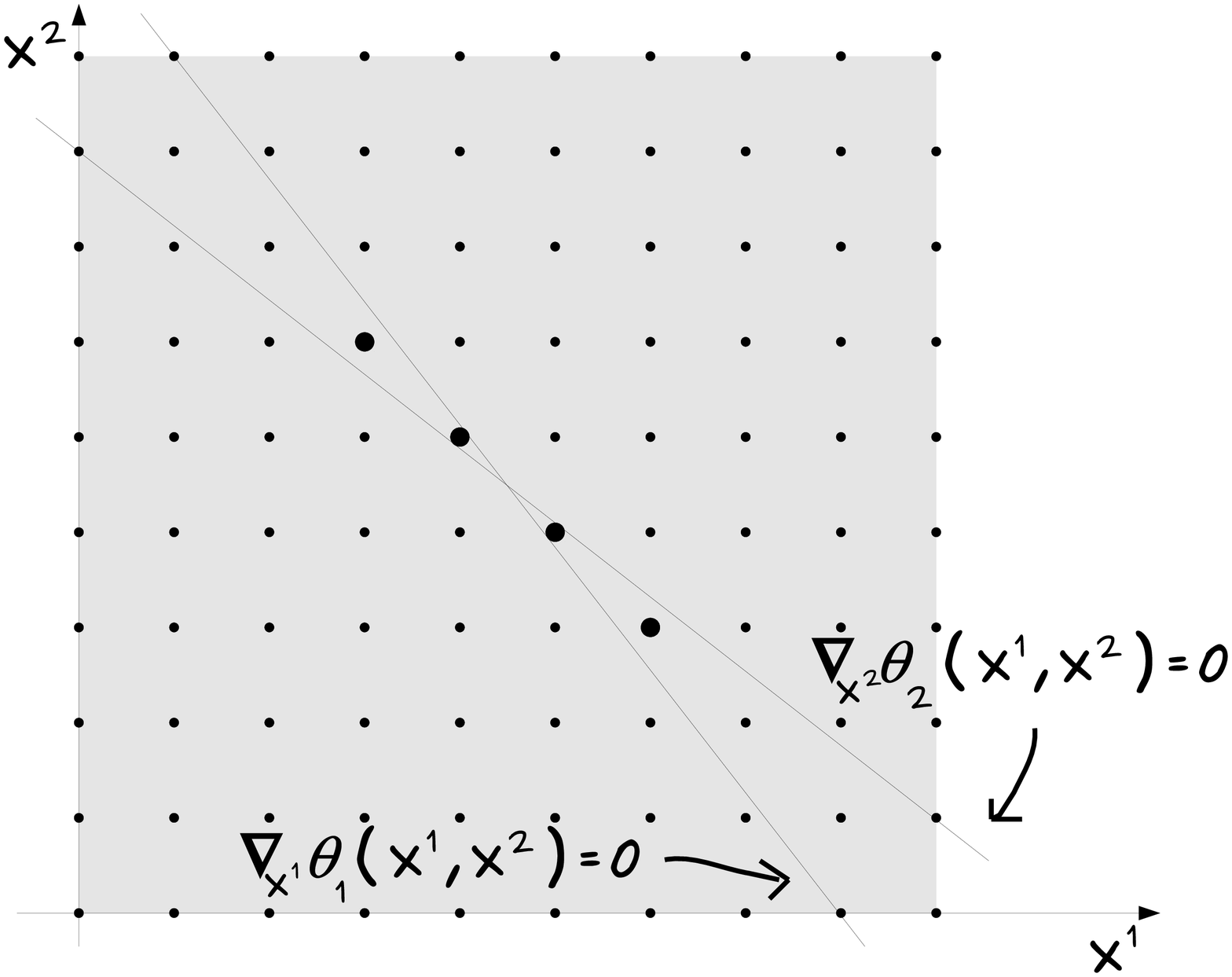}
 \caption{Equilibria for Example \ref{ex: one}.}
 \label{fig: ex1}
 \end{minipage}
 \begin{minipage}{0.49\textwidth}
 \centering
 \includegraphics[width=1\textwidth]{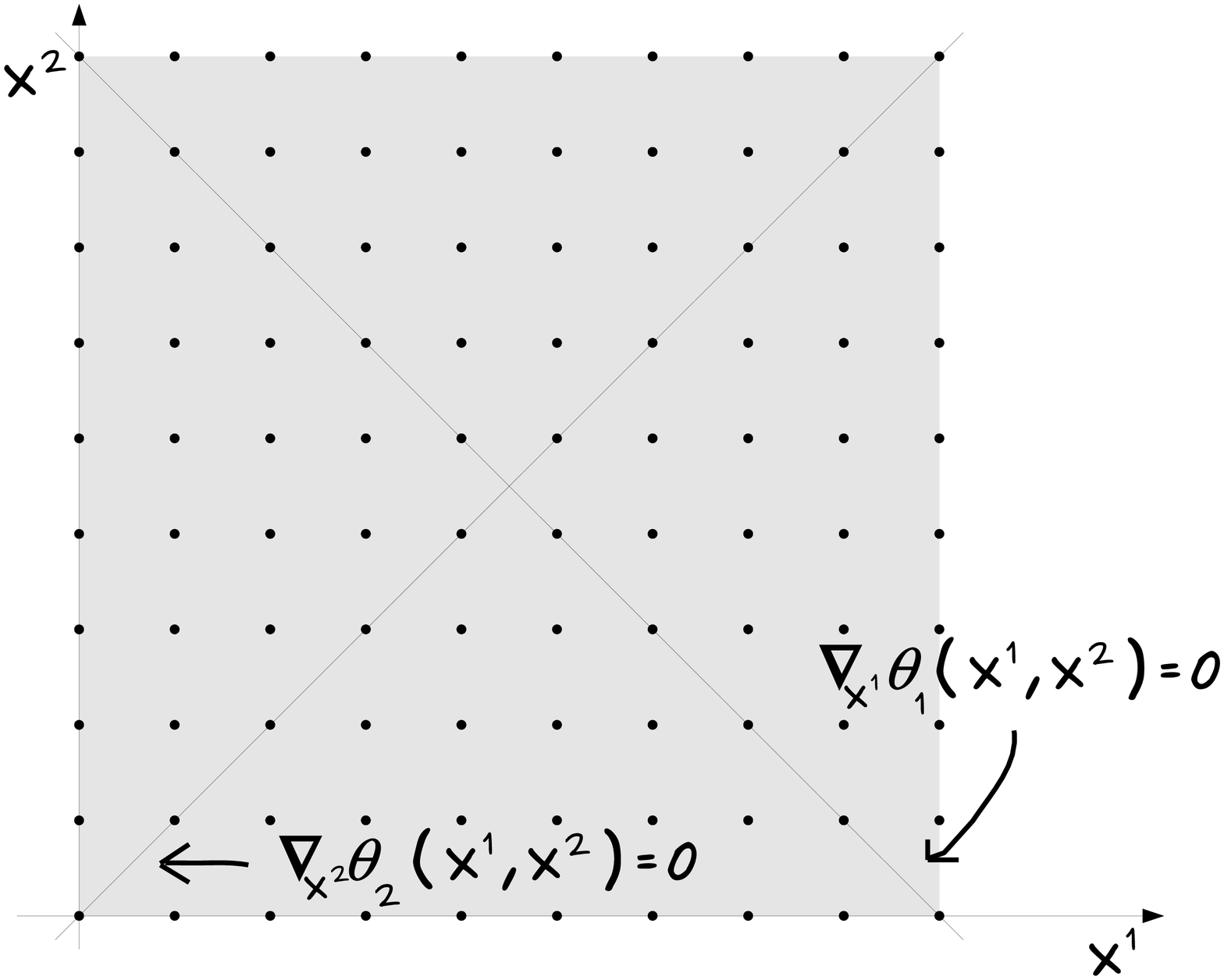}
 \caption{Lack of equilibria in Example \ref{ex: two}.}
 \label{fig: ex2}
\end{minipage}
\end{figure}
\noindent
Being $JF(x^1,x^2) = \left(\begin{array}{cc} 9 & 7 \\ 7 & 9 \end{array}\right)$ symmetric, function $f(x^1,x^2)=\frac{9}{2} (x^1)^2 + 7 x^1 x^2 + \frac{9}{2} (x^2)^2 - 72 x^1 - 72 x^2$ is such that $\nabla_{(x^1,x^2)} f(x^1,x^2) = F(x^1,x^2)$ for all $\xbold \in X$. However, the optimization problem defined by
$$
\begin{array}{c}
 \displaystyle\min_{(x^1,x^2)} \, f(x^1,x^2) = \frac{9}{2} (x^1)^2 + 7 x^1 x^2 + \frac{9}{2} (x^2)^2 - 72 x^1 - 72 x^2\\
\text{    } \quad 0 \le x^1 \le 9\\
\text{    } \quad 0 \le x^2 \le 9\\
\text{    } \quad \quad (x^1,x^2) \in \Z^2,
\end{array}
$$
has the following solutions: $(4,5)$ and $(5,4)$, but not $(3,6)$ and $(6,3)$.
% \hspace*{\fill}$\Box$
\end{example}
\medskip

\noindent
In any case, it is straightforward to prove the contrary, i.e. that all solutions of discrete optimization problem \eqref{eq: optimization problem} are equilibria for the discrete NEP defined by \eqref{eq: prob}.
And, therefore, we can say that a necessary, but not sufficient, condition for a point $\bar \xbold$ to be a solution of problem \eqref{eq: optimization problem} is to be a solution of the discrete NEP defined by \eqref{eq: prob}.

Reformulating a dicrete NEP by using KKT conditions or variational inequalities is not fruitful, since first order optimality conditions cannot be effectively used with discrete strategy sets. However it is easy to see that relaxing integrality and then solving the resulting continuous NEP, defined by \eqref{eq: continuous prob}, may produce an integer solution, which then is also a solution for the original discrete game defined by \eqref{eq: prob}. The following proposition, whose proof is trivial and therefore omitted, formalizes this issue.
\begin{proposition}\label{pro: relaxed solution}
 Let $\ybold_*\in X \cap \Z^n$ be such that for all $\nu \in \{1, \ldots, N\}$: 
 $$
 \theta_\nu (y^\nu_*, \ybold^{-\nu}_*) \leq \theta_\nu (x^\nu, \ybold^{-\nu}_*), \qquad \forall \, x^\nu \in X^\nu,
 $$
 that is $\ybold_*$ is a solution of the continuous NEP.
 Then it holds that for all $\nu \in \{1, \ldots, N\}$: 
 $$
 \theta_\nu (y^\nu_*, \ybold^{-\nu}_*) \leq \theta_\nu (x^\nu, \ybold^{-\nu}_*), \qquad \forall \, x^\nu \in X^\nu \cap \Z^{n_\nu},
 $$
 that is $\ybold_*$ is a solution of the discrete NEP.
 In particular, we say that $\ybold_*$ is a \emph{favorable solution} for the discrete NEP.
\end{proposition}
\medskip

\noindent
In general, favorable solutions (defined in Proposition \ref{pro: relaxed solution}) are only a subset of the solution set of a discrete NEP. Moreover, very often, a discrete NEP can have more than one solution, but not favorable ones. In Example \ref{ex: one} this occurs, in that, none equilibrium of the discrete NEP is a favorable solution since, by relaxing the integrality contraints, the corresponding continuous NEP has the unique solution $(4.5,4.5) \notin \Z^2$, see Figure \ref{fig: ex1}.

It is not easy to give conditions ensuring the existence of solutions for discrete NEPs in their general form.
In fact, unlike for continuous problems, neither compactness of feasible set $X$, nor strong monotonicity of operator $F$, nor both together, can guarantee the existence of at least one solution, see Example \ref{ex: two}.
\begin{example}\label{ex: two}
There are two
players each controlling one variable. Players' problems are
$$
\begin{array}{cccc}
 \displaystyle\min_{x^1} \, \theta_1(x^1,x^2) = \frac{1}{2} (x^1)^2 + x^1 x^2 - 9 x^1 & & \quad &
\displaystyle\min_{x^2} \, \theta_2(x^1,x^2) = \frac{1}{2} (x^2)^2 - x^1 x^2\\
\text{   } \quad 0 \le x^1 \le 9 & & \quad & 
\quad\text{   } \quad 0 \le x^2 \le 9 \\
\text{    } \quad x^1 \in \Z  & & \quad & 
\quad\text{    } \quad x^2 \in \Z.
\end{array}
$$

\noindent
Note that $0 \le x^1 \le 9$ and $0 \le x^2 \le 9$ define a compact set and that operator $F$ is strongly monotone.
However the problem does not have any equilibrium, see Figure \ref{fig: ex2}.% \hspace*{\fill}$\Box$
\end{example}
\medskip

\noindent
The most obvious way in order to give sufficient conditions for the existence of equilibria of a discrete NEP is to prove that at least one favorable solution exists.
Following this reasoning, some researchers had considered discrete NEPs, then written down their KKT conditions, by previously relaxing integrality constraints, and finally given conditions ensuring that at least one of these KKT points is integer, see e.g. \cite{gabriel2013solving}.
As said above, favorable solutions are only a subset (often empty) of the whole solution set of a discrete NEP, therefore it is worth to give conditions for the existence of solutions without assuming that these are favorable ones.
We have to cite the works of Yang et al., \cite{van2007computing,yang2008solutions}, that, by developing a theory on discrete nonlinear complementarity problems, proposed an alternative way to guarantee the existence of at least one equilibrium of discrete NEPs coming from some economic applications. However, results given in \cite{yang2008solutions} are quite technical, and is rather difficult to use them in order to define general classes of problems for which existence can be proven.
One important contribution on this topic was given by Topkis, \cite{topkis1998supermodularity}, that, by studying supermodular games, defined a class of discrete NEPs that have at least one solution.
In Section \ref{sec:class}, we give conditions for the existence of equilibria (not only favorable ones) for a new class of discrete NEPs and we compare with results in \cite{topkis1998supermodularity} and in \cite{yang2008solutions}.

Notation: $M \in \mathbb{M}_{m \times n}$ is a matrix with $m$ rows and $n$ columns; $M_{j*}$ denotes the $j$-th row of $M$ and $M_{*i}$ denotes the $i$-th column of $M$; given a set of row indices $J_r$ and a set of column indices $J_c$, $M_{J_r J_c}$ is the submatrix with rows in $J_r$ and columns in $J_c$.
\section{A branching method for finding all equilibria}\label{sec: branching}

In this section we present a method for finding all solutions of the discrete NEP defined by \eqref{eq: prob}.
This method mimics the branch and bound paradigm for discrete and combinatorial optimization problems. It consists of a systematic enumeration of candidate solutions by means of strategy space search: the set of candidate solutions is thought of as forming a rooted tree with the full set at the root. The algorithm explores branches of this tree, and during the enumeration of the candidate solutions a procedure is used in order to reduce the strategy set, in particular by fixing some variables.

It is important to say that bounding strategies, which are standard for discrete and combinatorial optimization problems, in general cannot be easily applied for discrete NEPs because of the difficulty to compute a lower bound for a suitable merit function.
Standard merit fuctions for contiuous NEPs, see e.g. \cite{FacchPangBk}, not only cannot be directly applied in a discrete framework, but neither are easy to optimize.
So, computing a lower bound for a merit function in a subset of the feasible discrete region needs, except in special cases, a complete enumeration.
The situation is obviously much easier when $JF(\xbold)$ is symmetric for all $\xbold \in X$ since, as said before, a solution of the discrete NEP can be found by solving problem \eqref{eq: optimization problem}, and then $f(\xbold)$ can be an ``easy'' merit function. 
A similar situation also occurs in potential games, see \cite{FacchPiccScia11}.
However here we consider the general case and therefore given a subset of the strategy space it is very difficult to say if it cannot contain an equilibrium by using a standard bounding method.

The method we propose is defined below, but, before, we have to define some tools that it uses:
\begin{itemize}
 \item an oracle ${\cal O}$ that takes a point $\bar \xbold \in X \cap \Z^n$ and outputs YES if and only if $\bar \xbold$ is an equilibrium for the discrete NEP defined by \eqref{eq: prob};
 \item a procedure ${\cal S}$ that, given a convex set $Y$ such that
 \begin{equation}\label{eq: set Y}
  Y^\nu \subseteq X^\nu, \qquad Y := \prod_{\nu=1}^N Y^\nu,
 \end{equation}
 yields one equilibrium of the continuous NEP in which each player $\nu$ solves
 \begin{equation}\label{eq: continuous prob Y}
 \begin{array}{l}
  \min_{x^\nu} \theta_\nu (x^\nu, \xbold^{-\nu}) \\[0.5em]
  \quad x^\nu \in Y^\nu;
 \end{array}
 \end{equation}
 \item a procedure ${\cal F}$ that, given set $Y$ defined in \eqref{eq: set Y} and a point $\bar \xbold$ returned by procedure ${\cal S}$, yields a closed, possibly unbounded, box $B$ such that set $Y \setminus B$ does not contain any equilibrium of the discrete NEP defined by \eqref{eq: prob};
 \item a procedure ${\cal C}$ that, given set $Y$ defined in \eqref{eq: set Y} and a point $\bar \xbold \in Y \cap \Z^n$, yields $p$ closed, maybe unbounded, boxes $\bar B_i$ such that $\bar \xbold \notin \cup_{i = 1}^p \bar B_i$, that $\bar B_i \cap \bar B_j = \emptyset$ for all $i \neq j$, and that $\tilde \xbold \in \cup_{i = 1}^p \bar B_i$ for all $\tilde \xbold \in Y \cap \Z^n$, $\tilde \xbold \neq \bar \xbold$.
\end{itemize}
Later in this section we will discuss more in detail about these tools, now we are ready to define the branching method for finding all solutions of the discrete NEP defined by \eqref{eq: prob}.

\begin{algorithm}\label{alg: branching} (Branching Method)
 \begin{description}
  \item[(S.0)] Initialize list of strategy subsets ${\cal L} := \{ X \}$ and set of equilibria $E := \emptyset$.
  \item[(S.1)] Take from ${\cal L}$ a strategy set $Y$.
  \item[(S.2)] Use ${\cal S}$ to obtain a solution $\bar \xbold$ of the continuous NEP defined by \eqref{eq: continuous prob Y}.
  \item[(S.3)] Use ${\cal F}$, with $Y$ and $\bar \xbold$, to obtain box $B$.
  \item[(S.4)] If $\bar \xbold \in \Z^n$ then use ${\cal O}$ on $\bar \xbold$: if ${\cal O}$ says YES then put $\bar \xbold$ in $E$.
  \item[(S.5)] If $\bar \xbold \in \Z^n$ then use ${\cal C}$, with $Y$ and $\bar \xbold$, to obtain $p$ boxes $\bar B_i$;\\
  for all $i \in \{1, \ldots, p\}$, if $Y \cap B \cap \bar B_i \neq \emptyset$ then put it in ${\cal L}$;\\
  go to \textbf{\emph{(S.7)}}.
  \item[(S.6)] Find an index $i \in \{1, \ldots, n\}$ such that $\bar \xbold_i \notin \Z$;\\
  if $Y \cap B \cap \{\xbold \in \Re^n: \xbold_i \geq \lceil \bar \xbold_i \rceil \} \neq \emptyset$ then put it in ${\cal L}$;\\
  if $Y \cap B \cap \{\xbold \in \Re^n: \xbold_i \leq \lfloor \bar \xbold_i \rfloor \} \neq \emptyset$ then put it in ${\cal L}$.
  \item[(S.7)] If ${\cal L} = \emptyset$ then STOP, else go to \textbf{\emph{(S.1)}}.
 \end{description}
\end{algorithm}
\medskip

\noindent
Eventually Algorithm \ref{alg: branching} enumerates all points in $X \cap \Z^n$, checking their optimality, except those that are cut off by using procedure ${\cal F}$. Therefore, it is clear that: (i) if $X$ is compact, Algorithm \ref{alg: branching} computes the whole solution set of the discrete NEP and (ii) procedure ${\cal F}$ is crucial in order to obtain an efficient algorithm.

On the other hand, if we are interested in computing only one equilibrium of the discrete NEP, and not the whole solution set, then we can stop Algorithm \ref{alg: branching} as soon as it finds a solution, so considerably increasing efficiency of the method.

Now we describe in detail the tools used by Algorithm \ref{alg: branching}.

\subsection{Oracle $\cal O$}
Given a point $\bar \xbold \in \Z^n$, let us define the following best response at $\bar \xbold$ for each player $\nu$:
\begin{equation}\label{eq: best response}
 \hat x^\nu(\bar \xbold^{-\nu}) := \arg \min_{x^\nu \in X^\nu \cap \Z^{n_\nu}} \theta_\nu (x^\nu, \bar \xbold^{-\nu}).
\end{equation}
Oracle $\cal O$ must certify optimality of a point $\bar \xbold \in \Z^n$.
Therefore, for all $\nu$, it checks if $\bar x^\nu \in \hat x^\nu({\bar \xbold}^{-\nu})$, and, if it is true, then answers YES, otherwise answers NO. 
Note that, in practice, computing a point in $\hat x^\nu({\bar \xbold}^{-\nu})$ could be a demanding task, since, in general, it requires to use a Mixed Integer Non-Linear Programming tool, see e.g. \cite{belotti2013mixed,belotti2009branching,nowak2006relaxation,tawarmalani2002convexification}. In Section \ref{sec: numerical} we give specific implementation details on this issue.

\subsection{Procedure $\cal S$}
There are a lot of methods for finding a solution of a continuous NEP, see e.g. \cite{dirkse1995path,dreves2011solution,FacchPangBk}.
It is well known that if operator $F$ is monotone, or something a bit weaker, in the strategy set then there is more than one algorithm which is globally convergent to a solution of the continuous NEP. In Section \ref{sec: numerical} we give specific implementation details also on this issue.

\subsection{Procedure $\cal F$}\label{subsec: proocedure F}
As said above, procedure ${\cal F}$ is crucial in order to obtain an efficient algorithm.
First of all, let us analyze a situation that may set a trap. Suppose that $\bar \xbold \in \Z^n$ is a solution of the continuous NEP defined by \eqref{eq: continuous prob Y}, being $\bar \xbold$ integer then it is a favorable solution for the discrete NEP with strategy set $Y \subseteq X$, see Proposition \ref{pro: relaxed solution}. Moreover suppose that oracle $\cal O$ says that $\bar \xbold$ is not a solution of the discrete NEP defined by \eqref{eq: prob}, which has strategy set $X \supseteq Y$. Then one could think that this is enough to say that $Y$ cannot contain any equilibrium of the discrete NEP defined by \eqref{eq: prob}. But this is not true in general. Example \ref{ex: trap} shows this in a very simple setting even with $F$ strongly monotone.

\begin{example}\label{ex: trap}
There are two
players each controlling one variable. Players' problems are
$$
\begin{array}{cccc}
 \displaystyle\min_{x^1} \, \theta_1(x^1,x^2) = \frac{7}{16} (x^1)^2 - x^1 x^2 + \frac{1}{2} x^1 & & \quad &
\displaystyle\min_{x^2} \, \theta_2(x^1,x^2) = \frac{1}{2} (x^2)^2 - \frac{3}{4} x^1 x^2\\
\text{   } \quad -1 \le x^1 \le 2 & & \quad & 
\quad\text{   } \quad -1 \le x^2 \le 2 \\
\text{    } \quad x^1 \in \Z  & & \quad & 
\quad\text{    } \quad x^2 \in \Z.
\end{array}
$$

%\medskip
\begin{figure}[ht!]
 \centering
 \includegraphics[width=.6\textwidth]{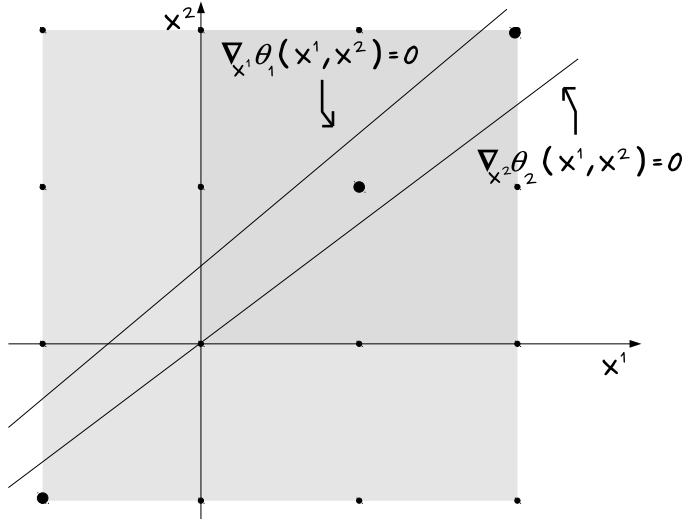}
 \caption{Equilibria for Example \ref{ex: trap}.}
 \label{fig: extrap}
\end{figure}

\noindent
Then $X = \{x^1, x^2 \in \Re: -1 \le x^1 \le 2, -1 \le x^2 \le 2 \}$ and it holds that
$$
JF(x^1,x^2) = \left(\begin{array}{cc} \;\;\, \frac{7}{8} & -1 \\[.5em] -\frac{3}{4} & \;\;\, 1 \end{array}\right) \succ 0.
$$
Let us indicate by $S \subseteq X \cap \Z^n$ the set of all solutions of this discrete NEP.
Now let us consider $Y = \{x^1, x^2 \in \Re: 0 \le x^1 \le 2, 0 \le x^2 \le 2 \} \subset X$ and indicate by $\bar S \subseteq X$ the set of all solutions of the continuous NEP with strategy set $Y$ and with the same objective functions.
Then it is easy to see that $S = \{(-1,-1), (1,1), (2,2)\}$ and $\bar S = \{(0,0)\}$, see Figure \ref{fig: extrap}.
Note that: $(0,0) \notin S$ and $S \supset \{(1,1), (2,2)\} \subset Y$.
%\hspace*{\fill}$\Box$
\end{example}
\medskip

\noindent
Therefore procedure ${\cal F}$ cannot use this simple strategy in order to indentify a subset which  provably does not contain any equilibrium. Procedure ${\cal F}$ we propose is motivated by the following proposition.

\begin{proposition}\label{pro: cut}
 Suppose that $X$ is defined by box constraints:
 \begin{equation}\label{eq: box}
 X := \left\{ \xbold \in \Re^n: \lbold \leq \xbold \leq \ubold \right\}, \qquad \lbold, \ubold \in \Z^n.  
 \end{equation}
 Let $\bar \xbold \in X \cap \Z^n$ be a solution for the continuous NEP defined by \eqref{eq: continuous prob}. 
 Let us consider a generic player $\nu$. Suppose that an index $i \in \{1, \ldots, n_\nu\}$ exists such that one of the following two possibilities holds:
 \begin{enumerate}[(i)]
  \item $F^\nu_i (\bullet)$ is a convex function,
  $\bar x^\nu_i = l^\nu_i$,  and
  for each player $\mu \in \{1, \ldots, N\}$ and each index $j \in \{1, \ldots, n_\mu\}$, such that $(\mu,j) \neq (\nu,i)$:
  \begin{itemize}
  \item if $\frac{\partial F^\nu_i (\bar \xbold)}{\partial x^\mu_j} > 0$ then $\bar x^\mu_j = l^\mu_j$,
  \item if $\frac{\partial F^\nu_i (\bar \xbold)}{\partial x^\mu_j} < 0$ then $\bar x^\mu_j = u^\mu_j$.
  \end{itemize}

  \item $F^\nu_i (\bullet)$ is a concave function,
  $\bar x^\nu_i = u^\nu_i$ and
  for each player $\mu \in \{1, \ldots, N\}$ and each index $j \in \{1, \ldots, n_\mu\}$, such that $(\mu,j) \neq (\nu,i)$:
  \begin{itemize}
  \item if $\frac{\partial F^\nu_i (\bar \xbold)}{\partial x^\mu_j} > 0$ then $\bar x^\mu_j = u^\mu_j$,
  \item if $\frac{\partial F^\nu_i (\bar \xbold)}{\partial x^\mu_j} < 0$ then $\bar x^\mu_j = l^\mu_j$.
  \end{itemize}
 \end{enumerate}
 If $\theta_\nu$ is strictly convex with respect to $x^\nu_i$,
 then any point $\tilde \xbold \in X \cap \Z^n$ such that $\tilde x^\nu_i \neq \bar x^\nu_i$ cannot be an equilibrium for the discrete NEP defined by \eqref{eq: prob}.
\end{proposition}
\begin{proof}
 Let $\bar{\tilde{\xbold}} \in X \cap \Z^n$ be such that $\bar{\tilde{x}}^\nu_i = \bar x^\nu_i$ and $\bar{\tilde{x}}^\mu_j = \tilde x^\mu_j$ for all $(\mu,j) \neq (\nu,i)$.
 And let $\tilde{\bar{\xbold}} \in X \cap \Z^n$ be such that $\tilde{\bar{x}}^\nu_i = \tilde x^\nu_i$ and $\tilde{\bar{x}}^\mu_j = \bar x^\mu_j$ for all $(\mu,j) \neq (\nu,i)$.

 Let us assume that situation (i) holds:
 function $F^\nu_i (\bullet)$ is convex, then we can write
 $$
  F^\nu_i (\bar{\tilde{\xbold}}) - F^\nu_i (\bar \xbold) \geq J F^\nu_i (\bar \xbold) (\bar{\tilde{\xbold}} - \bar \xbold).
 $$
 And then, by assumptions made in (i), it holds that
 \begin{equation}\label{eq: augmenting gradient}
  F^\nu_i (\bar \xbold) \leq F^\nu_i (\bar{\tilde{\xbold}}).
 \end{equation}
 The following chain of inequalities holds
 \begin{align}\label{eq: chain first theorem}
  0 \overset{(A)}{\leq} F^\nu (\bar \xbold)\trt (\tilde{\bar{\xbold}} - \bar \xbold) = F^\nu_i (\bar \xbold) (\tilde{x}^\nu_i - \bar x^\nu_i) \overset{(B)}{\leq} F^\nu_i (\bar{\tilde{\xbold}}) (\tilde{x}^\nu_i - \bar x^\nu_i) \overset{(C)}{<} \theta_\nu(\tilde \xbold) - \theta_\nu(\bar{\tilde{\xbold}}),
 \end{align}
 where (A) holds since $\bar \xbold$ is a solution for the continuous NEP and $\tilde{\bar{\xbold}}$ is a feasible point; (B) follows from \eqref{eq: augmenting gradient} and the fact that $\tilde x^\nu_i > \bar x^\nu_i$ since $\bar x^\nu_i = l^\nu_i$; (C) holds since $\theta_\nu$ is strictly convex with respect to $x^\nu_i$.
 Therefore $\theta_\nu(\tilde \xbold) > \theta_\nu(\bar{\tilde{\xbold}}) \in X \cap \Z^n$.

 Now let us assume that situation (ii) holds:
 function $F^\nu_i (\bullet)$ is concave, then we can write
 $$
  F^\nu_i (\bar{\tilde{\xbold}}) - F^\nu_i (\bar \xbold) \leq J F^\nu_i (\bar \xbold) (\bar{\tilde{\xbold}} - \bar \xbold).
 $$
 And then, by assumptions made in (ii), it holds that
 \begin{equation}\label{eq: augmenting gradient2}
  F^\nu_i (\bar \xbold) \geq F^\nu_i (\bar{\tilde{\xbold}}).
 \end{equation}
 Then, by using \eqref{eq: augmenting gradient2} and with the same rationale of case (i), chain of inequalities \eqref{eq: chain first theorem} holds. Therefore we have the proof.
\end{proof}
As said above procedure ${\cal F}$ is based on Proposition \ref{pro: cut}. Therefore it can be applied only when Algorithm \ref{alg: branching} faces situations satisfying assumptions of Proposition \ref{pro: cut}.
Let us assume, for simplicity, that
\begin{equation}\label{eq: quadratic functions and box constraints}
 \theta_\nu(\xbold) := \frac{1}{2} ({x^\nu})\trt Q^\nu x^\nu + \left(C^{\nu} \xbold^{-\nu} + b^\nu\right)\trt x^\nu, \quad \forall\, \nu \in \{1, \ldots, N\},
\end{equation}
with $Q^\nu \in {\mathbb M}_{n_\nu \times n_\nu}$, $C^\nu \in {\mathbb M}_{n_\nu \times (n - n_{\nu})}$ and $b^\nu \in \Re^{n_\nu}$ for all $\nu$. Moreover, let us assume that $X$ is defined as in \eqref{eq: box}. Then it is clear that, considering Algorithm \ref{alg: branching}, any strategy set $Y \in \cal L$ is a box.
In this case, and by exploiting Proposition \ref{pro: cut}, procedure ${\cal F}$ can be defined as follows:
\begin{algorithm}\label{alg: procedure F} (Procedure ${\cal F}$)
 \begin{description}
  \item[(Data)] A box strategy set $Y := \{\xbold \in \Re^n: \lbold \leq \xbold \leq \ubold\}$, with $\lbold \in \Z^n$ and $\ubold \in \Z^n$, and a solution $\bar \xbold$ of the continuous NEP defined by \eqref{eq: continuous prob Y}.
  \item[(S.0)] Set $B := \Re^n$, $\nu := 1$ and $i := 1$.
  \item[(S.1)] If all the following conditions hold:
  \begin{itemize}
   \item $Q^\nu_{ii} > 0$;
   \item $\bar x^\nu_i = l^\nu_i$;
   \item $\bar x^\nu_j = l^\nu_j$ for all $Q^\nu_{ij} > 0$, $j = 1, \ldots, n_\nu$, $j \neq i$;
   \item $\bar x^\nu_j = u^\nu_j$ for all $Q^\nu_{ij} < 0$, $j = 1, \ldots, n_\nu$, $j \neq i$;
   \item $\bar \xbold^{-\nu}_j = \lbold^{-\nu}_j$ for all $C^\nu_{ij} > 0$, $j = 1, \ldots, (n-n_\nu)$;
   \item $\bar \xbold^{-\nu}_j = \ubold^{-\nu}_j$ for all $C^\nu_{ij} < 0$, $j = 1, \ldots, (n-n_\nu)$;
  \end{itemize}
  then set $B := B \cap \{\xbold \in \Re^n: x^\nu_i \leq l^\nu_i \}$ and go to \textbf{\emph{(S.3)}}.
  \item[(S.2)] If all the following conditions hold:
  \begin{itemize}
   \item $Q^\nu_{ii} > 0$;
   \item $\bar x^\nu_i = u^\nu_i$;
   \item $\bar x^\nu_j = u^\nu_j$ for all $Q^\nu_{ij} > 0$, $j = 1, \ldots, n_\nu$, $j \neq i$;
   \item $\bar x^\nu_j = l^\nu_j$ for all $Q^\nu_{ij} < 0$, $j = 1, \ldots, n_\nu$, $j \neq i$;
   \item $\bar \xbold^{-\nu}_j = \ubold^{-\nu}_j$ for all $C^\nu_{ij} > 0$, $j = 1, \ldots, (n-n_\nu)$;
   \item $\bar \xbold^{-\nu}_j = \lbold^{-\nu}_j$ for all $C^\nu_{ij} < 0$, $j = 1, \ldots, (n-n_\nu)$;
  \end{itemize}
  then set $B := B \cap \{\xbold \in \Re^n: x^\nu_i \geq u^\nu_i \}$.
  \item[(S.3)] Set $i := i + 1$; if $i \leq n_\nu$ then go to \textbf{\emph{(S.1)}}, else set $i := 1$.
  \item[(S.4)] Set $\nu := \nu + 1$; if $\nu \leq N$ then go to \textbf{\emph{(S.1)}}.
  \item[(Output)] Return $B$.
  \end{description}
\end{algorithm}

\subsection{Procedure $\cal C$}
This procedure is used in order to produce $p$ boxes $\bar B_i$ such that, given a point $\bar \xbold \in \Z^n$ and a strategy set $Y$, we obtain $\bar \xbold \notin \cup_{i = 1}^p \bar B_i$, and $\bar B_i \cap \bar B_j = \emptyset$ for all $i \neq j$, and $\tilde \xbold \in \cup_{i = 1}^p \bar B_i$ for all $\tilde \xbold \in Y \cap \Z^n$ with $\tilde \xbold \neq \bar \xbold$.

There are a lot of different implementations for this procedure. In this work we propose the following which yields $p=2n$ boxes.
\begin{algorithm}\label{alg: procedure C} (Procedure ${\cal C}$)
 \begin{description}
  \item[(Data)] A point $\bar \xbold \in \Z^n$.
  \item[(S.0)] Set $\bar B_i := \Re^n$, for all $i \in \{1, \ldots, 2n\}$, and set $j := 1$.
  \item[(S.1)] Set:
  \begin{itemize}
   \item $\bar B_{2j-1} := \bar B_{2j-1} \cap \{\xbold \in \Re^n: \, \xbold_j \geq \bar \xbold_j + 1 \}$;
   \item $\bar B_{2j} := \bar B_{2j} \cap \{\xbold \in \Re^n: \, \xbold_j \leq \bar \xbold_j - 1 \}$;
   \item $\bar B_{t} := \bar B_{t} \cap \{\xbold \in \Re^n: \, \xbold_j = \bar \xbold_j \}$ for all $t \in \Z$ such that $2j+1 \leq t \leq 2n$.
  \end{itemize}
  \item[(S.2)] Set $j := j + 1$; if $j \leq n$ then go to \textbf{\emph{(S.1)}}.
  \item[(Output)] Return $\bar B_i$, $i = 1, \ldots, 2n$.
  \end{description}
\end{algorithm}
At this point, we have described all tools used by Algorithm \ref{alg: branching}, which, as said above, can be used to compute the whole solution set of any discrete NEP. Its efficiency is totally based on procedure $\cal F$ that prunes the branches of the search tree. In the next subsection, we propose an improved version of the branching algorithm, that, by using a preliminary procedure, shrinks the search area and drastically reduces the feasible points to be examined. However this improved algorithm can be used only if the problem has a particular structure.

\subsection{An improved version of Algorithm \ref{alg: branching}}\label{subsec: algorithmic framework}

Here we propose an efficient algorithmic framework for finding all equilibria of discrete NEPs in which functions $\theta_\nu$ are defined as in \eqref{eq: quadratic functions and box constraints} and $X$ is defined by box constraints as in \eqref{eq: box}.
Note that many noncooperative games can be modeled as NEPs defined by \eqref{eq: quadratic functions and box constraints} and \eqref{eq: box}, see e.g. \cite{basar1995dynamic}.
This framework is composed of three parts: (i) compute lower bounds $\lbold^* \geq \lbold$ for the set of solutions, (ii) compute upper bounds $\ubold^* \leq \ubold$ for the set of solutions, and, then, (iii) use Algorithm \ref{alg: branching} in order to get all solutions by exploring points in $[\lbold^*,\ubold^*] \cap \Z^n$.
The framework presented in this subsection is a great deal faster than just Algorithm \ref{alg: branching} to compute either one equilibrium or the whole solution set of the discrete NEP. This is due to the effectiveness of the shrinkage of the search area: from 
$[\lbold,\ubold]$ to $[\lbold^*,\ubold^*]$.

The following Gauss-Seidel method performs the first task, that is computing lower bounds $\lbold^*$ for the set of solutions of the discrete NEP.
\begin{algorithm}\label{alg: lower bound} (Gauss-Seidel Method to compute Solution Set Lower Bounds)
 \begin{description}
  \item[(S.0)] Set $\wbold := \lbold$ and $\ybold := \lbold$.
  \item[(S.1)] Set $\nu:=1$.
  \item[(S.2)] Set $i:=1$.
  \item[(S.3)] Set $\zbold \in \Z^n$ such that, for all $\mu \in \{1,\ldots,N\}$ and all $j \in \{1,\ldots,n_\mu\}$,
   $z^\mu_j := y^\mu_j$ if $\frac{\partial F^\nu_i}{\partial x^\mu_j} \leq 0$
   and $z^\mu_j := u^\mu_j$ otherwise.
  \item[(S.4)] Compute $s \in \Z$ such that $y^\nu_i \leq s \leq u^\nu_i$,
  
  if $s+1 \leq u^\nu_i$ then
  \begin{equation}\label{eq: solsetlowbnd cond1}
  \theta_\nu \left( \left(
  \begin{array}{c} z^\nu_1 \\ \vdots \\ z^\nu_{i-1} \\[1em] s \\[1em] z^\nu_{i+1} \\ \vdots \\ z^\nu_{n_\nu} \end{array}\right),
  \zbold^{-\nu} \right)
  \leq
  \theta_\nu \left( \left(
  \begin{array}{c} z^\nu_1 \\ \vdots \\ z^\nu_{i-1} \\[1em] s+1 \\[1em] z^\nu_{i+1} \\ \vdots \\ z^\nu_{n_\nu} \end{array}\right),
  \zbold^{-\nu} \right),
  \end{equation}
  and if $y^\nu_i \leq s-1$ then
  \begin{equation}\label{eq: solsetlowbnd cond2}
  \theta_\nu \left( \left(
  \begin{array}{c} z^\nu_1 \\ \vdots \\ z^\nu_{i-1} \\[1em] s-1 \\[1em] z^\nu_{i+1} \\ \vdots \\ z^\nu_{n_\nu} \end{array}\right),
  \zbold^{-\nu} \right)
  >
  \theta_\nu \left( \left(
  \begin{array}{c} z^\nu_1 \\ \vdots \\ z^\nu_{i-1} \\[1em] s \\[1em] z^\nu_{i+1} \\ \vdots \\ z^\nu_{n_\nu} \end{array}\right),
  \zbold^{-\nu} \right).
  \end{equation}
  \item[(S.5)] Set $y^\nu_i := s$.
  \item[(S.6)] Set $i:=i+1$. If $i > n_\nu$ then set $\nu:=\nu+1$. 
  If $\nu \leq N$ then go to \textbf{\emph{(S.2)}}.
  \item[(S.7)] If $\wbold = \ybold$ then STOP and return $\lbold^* := \ybold$. Else set $\wbold := \ybold$ and go to \textbf{\emph{(S.1)}}.
 \end{description}
\end{algorithm}
\medskip
The following theorem states that point $\lbold^*$ returned by Algorithm \ref{alg: lower bound} is a lower bound for the set of solutions of the discrete NEP.

\begin{theorem}\label{th: lower bound convergence}
 Suppose that $\theta_\nu$ are defined as in \eqref{eq: quadratic functions and box constraints} and that $X$ is non-empty, bounded and defined as in \eqref{eq: box}.

 Then Algorithm \ref{alg: lower bound} returns point $\lbold^*$ in a finite number of iterations and for any solution $\xbold^*$ of the discrete NEP, defined by \eqref{eq: prob}, it holds that $\lbold^* \leq \xbold^*$.
\end{theorem}
\begin{proof}
 At {\bf (S.5)} of the first iteration of the algorithm,
 let us consider any point $\bar \xbold \in X \cap \Z^n$ such that $\bar x^1_1 < y^1_1$ and the corresponding point $\tilde \xbold \in X \cap \Z^n$ such that $\tilde x^\nu_i = \bar x^\nu_i$ for all $(\nu,i) \neq (1,1)$ and $\tilde x^1_1 = y^1_1$.

 By assumptions, it holds that
 \begin{equation}\label{eq: augmenting gradient lb}
  \displaystyle \sum_{j=2}^{n_1} Q^1_{1j} \bar x^1_j \leq \sum_{j=2}^{n_1} Q^1_{1j} z^1_j, \qquad \text{ and, } \qquad
  C^1_{1*} \bar \xbold^{-1} \leq C^1_{1*} \zbold^{-1}.
 \end{equation}
 The following chain of inequalities holds
 \begin{align*}
  \frac{1}{2} Q^1_{1\,1} \left(\bar x^1_1\right)^2 - \frac{1}{2} Q^1_{1\, 1} \left(y^1_1\right)^2 \overset{(A)}{>} & \\
  \left( \displaystyle \sum_{j=2}^{n_1} Q^1_{1j} z^1_j + C^1_{1*} \zbold^{-1} + b^1 \right) (y^1_1 - \bar x^1_1) \overset{(B)}{\geq} & \\
  \left( \displaystyle \sum_{j=2}^{n_1} Q^1_{1j} \bar x^1_j + C^1_{1*} \bar \xbold^{-1} + b^1 \right) (y^1_1 - \bar x^1_1),
 \end{align*}
 where (A) holds by \eqref{eq: solsetlowbnd cond2} and the convexity of $\theta_1$ with respect to $x^1_1$; (B) follows from \eqref{eq: augmenting gradient lb} and $\bar x^1_1 < y^1_1$.
 Therefore $\theta_1 (\bar \xbold) > \theta_1 (\tilde \xbold)$, and then $\bar \xbold$ cannot be an equilibrium of the discrete NEP.
 
 By iterating this reasoning, by using the new lower bounds as soon as they are computed, and by noting that $\ybold \geq \wbold$ for all iterations, we get the proof.
\end{proof}
The algorithm to compute upper bounds $\ubold^*$ for the set of solutions of the discrete NEP is specular. We report it for completeness.
\begin{algorithm}\label{alg: upper bound} (Gauss-Seidel Method to compute Solution Set Upper Bounds)
 \begin{description}
  \item[(S.0)] Set $\wbold := \ubold$ and $\ybold := \ubold$.
  \item[(S.1)] Set $\nu:=1$.
  \item[(S.2)] Set $i:=1$.
  \item[(S.3)] Set $\zbold \in \Z^n$ such that, for all $\mu \in \{1,\ldots,N\}$ and all $j \in \{1,\ldots,n_\mu\}$,
   $z^\mu_j := y^\mu_j$ if $\frac{\partial F^\nu_i}{\partial x^\mu_j} \leq 0$
   and $z^\mu_j := l^\mu_j$ otherwise.
  \item[(S.4)] Compute $s \in \Z$ such that $l^\nu_i \leq s \leq y^\nu_i$,
  
  if $l^\nu_i \leq s-1$ then
  \begin{equation}\label{eq: solsetupbnd cond1}
  \theta_\nu \left( \left(
  \begin{array}{c} z^\nu_1 \\ \vdots \\ z^\nu_{i-1} \\[1em] s-1 \\[1em] z^\nu_{i+1} \\ \vdots \\ z^\nu_{n_\nu} \end{array}\right),
  \zbold^{-\nu} \right)
  \geq
  \theta_\nu \left( \left(
  \begin{array}{c} z^\nu_1 \\ \vdots \\ z^\nu_{i-1} \\[1em] s \\[1em] z^\nu_{i+1} \\ \vdots \\ z^\nu_{n_\nu} \end{array}\right),
  \zbold^{-\nu} \right),
  \end{equation}
  and if $s+1 \leq y^\nu_i$ then
  \begin{equation}\label{eq: solsetupbnd cond2}
  \theta_\nu \left( \left(
  \begin{array}{c} z^\nu_1 \\ \vdots \\ z^\nu_{i-1} \\[1em] s \\[1em] z^\nu_{i+1} \\ \vdots \\ z^\nu_{n_\nu} \end{array}\right),
  \zbold^{-\nu} \right)
  <
  \theta_\nu \left( \left(
  \begin{array}{c} z^\nu_1 \\ \vdots \\ z^\nu_{i-1} \\[1em] s+1 \\[1em] z^\nu_{i+1} \\ \vdots \\ z^\nu_{n_\nu} \end{array}\right),
  \zbold^{-\nu} \right).
  \end{equation}
  \item[(S.5)] Set $y^\nu_i := s$.
  \item[(S.6)] Set $i:=i+1$. If $i > n_\nu$ then set $\nu:=\nu+1$. 
  If $\nu \leq N$ then go to \textbf{\emph{(S.2)}}.
  \item[(S.7)] If $\wbold = \ybold$ then STOP and return $\ubold^* := \ybold$. Else set $\wbold := \ybold$ and go to \textbf{\emph{(S.1)}}.
 \end{description}
\end{algorithm}
\medskip
Algorithm \ref{alg: upper bound} works under the same assumptions of Algorithm \ref{alg: lower bound} and then we skip a formal convergence result.

We are now ready to define the improved algorithm to compute the whole solution set of the discrete NEP.
\begin{algorithm}\label{alg: class} (Improved Branching Method)
 \begin{description}
  \item[(S.1)] Compute $\lbold^*$ by using Algorithm \ref{alg: lower bound}.
  \item[(S.2)] Compute $\ubold^*$ by using Algorithm \ref{alg: upper bound}.
  \item[(S.3)] Compute the whole solution set of the discrete NEP by using Algorithm \ref{alg: branching} and initializing ${\cal L} := \{[\lbold^*,\ubold^*]\}$.
 \end{description}
\end{algorithm}
In Section \ref{sec: numerical} we show that Algorithm \ref{alg: class} works very well in practice and effectively computes the whole solution set of discrete NEPs.

\section{Fast algorithms and existence results for a class of discrete NEPs}\label{sec:class}

In this section we define a class of discrete NEPs for which we can give stronger results, namely, existence of equilibria and a fast Jacobi-type method for computing one of their equilibria.

\begin{definition}\label{def: 2groups}
We say that a discrete NEP is \emph{2-groups partitionable} if it satisfies the following conditions:
\begin{enumerate}[(i)]
\item $X$ is defined by box constraints as in \eqref{eq: box};
\item $\theta_\nu$ is defined, for each player $\nu$, as
\begin{equation}\label{eq: objective quadratic generalized}
 \theta_\nu (x^\nu, \xbold^{-\nu}) := \theta^P_\nu (x^\nu) + \Theta^O_\nu (\xbold^{-\nu})\trt x^\nu,
\end{equation}
where
$\theta_\nu^P: \Re^{n_\nu} \to \Re$ is defined in the following way
\begin{equation}\label{eq: private quadratic generalized}
 \theta_\nu^P (x^\nu) := \sum_{i = 1}^{n_\nu} \vartheta_{\nu,i}^P (x^\nu_i) + \frac{1}{2} ({x^\nu})\trt Q^\nu x^\nu,
\end{equation}
$\vartheta_{\nu,i}^P : \Re \to \Re$, for all $i$, and $Q^\nu \in {\mathbb M}_{n_\nu \times n_\nu}$,
and $\Theta^O_\nu : \Re^{n_{-\nu}} \to \Re^{n_\nu}$ is an operator made up of convex or concave functions;
\item a partition of the variables indices in two groups, $G_1$ and $G_2$, exists such that
\begin{align}
& \frac{\partial F^\nu_i(\xbold)}{\partial x^\mu_j} \leq 0, \; \forall \, \xbold \in X, \; \forall \, (\nu,i) \neq (\mu,j) \, : \nonumber\\
& \qquad\qquad\qquad\qquad\qquad (\nu,i) \in G_1 \ni (\mu,j) \text{ or } (\nu,i) \in G_2 \ni (\mu,j), \label{eq: new Z condition 1}\\
& \frac{\partial F^\nu_i(\xbold)}{\partial x^\mu_j} \geq 0, \; \forall \, \xbold \in X, \; \forall \, (\nu,i) \neq (\mu,j) \, : \nonumber\\
& \qquad\qquad\qquad\qquad\qquad (\nu,i) \in G_1 \not\ni (\mu,j) \text{ or } (\nu,i) \in G_2 \not\ni (\mu,j). \label{eq: new Z condition 2}
\end{align} 
\end{enumerate}
\end{definition}
The following is an example of a \emph{2-groups partitionable} discrete NEP.
\begin{example}\label{ex: class easy}
There are two
players. The first player solves the following problem
\begin{align*}
 \displaystyle\min_{x^1} \, \theta_1(x^1,x^2) & = \frac{1}{2}(x^1)\trt \left( \begin{array}{cc} 3 & 1\\1 & 3\end{array}\right)x^1 + (x^1)\trt \left( \begin{array}{cc} 4 & -3\\-1 & 1\end{array}\right)x^2                                                                                            + (7 \quad 2) x^1 \\
& \qquad \qquad -5 \le x^1 \le 5 \\
& \qquad \qquad \qquad x^1 \in \Z^2,
\end{align*}
while the second player solves the following problem
\begin{align*}
 \displaystyle\min_{x^2} \, \theta_2(x^1,x^2) & = \frac{1}{2}(x^2)\trt \left( \begin{array}{cc} 2 & 1\\1 & 2\end{array}\right)x^2 + (x^2)\trt \left( \begin{array}{cc} 1 & -2\\-3 & 4\end{array}\right)x^1                                                                                            + (5 \quad 6) x^2 \\
& \qquad \qquad -5 \le x^2 \le 5 \\
& \qquad \qquad \qquad x^2 \in \Z^2.
\end{align*}
By partitioning variables in this way: $G_1 = \{(1,1), (2,2)\}$, $G_2 = \{(1,2), (2,1)\}$, then \eqref{eq: new Z condition 1} and \eqref{eq: new Z condition 2} are satisfied.
\end{example}
Note that also the problems in Examples \ref{ex: one} and \ref{ex: trap} are \emph{2-groups partitionable} discrete NEPs.

In order to give an economic interpretation of condition (iii) in Definition \ref{def: 2groups}, let us consider a standard pricing game.
There are $N$ firms that compete in the same market in order to increase their profits as much as possible. Each firm $\nu$ produces a single product and sets its price $p^\nu \in \Z$. The prices are considered as integer, rather than real, variables in order to get the model more realistic, since the price of a product cannot be specified more closely than the minimum unit of a currency. For the sake of simplicity, let us assume that the consumers demand function for each firm $\nu$ is linear:
$$
D_\nu (p^\nu,\pbold^{-\nu}) := a_\nu - b_\nu p^\nu + (c^\nu)\trt \pbold^{-\nu}, 
$$
where $a_\nu, \, b_\nu \in \Re_+$ and $c^\nu \in \Re^{N-1}$. Moreover, assume that there are no fixed costs of production and marginal cost $d_\nu$ for each firm $\nu$ is such that $0 < d_\nu < a_\nu$. Therefore, for each firm $\nu$, since the objective is to maximize its profit $D_\nu (p^\nu,\pbold^{-\nu}) (p^\nu - d_\nu)$, the optimization problem to solve is the following:
\begin{align*}
\displaystyle\min_{p^\nu} & \; b_\nu (p^\nu)^2 - (\pbold^{-\nu})\trt c^\nu p^\nu - (a_\nu + b_\nu d_\nu) p^\nu\\
& 0 \leq p^\nu \leq P_\nu, \quad p^\nu \in \Z, 
\end{align*}
where $P_\nu \in \Z$ is a suitable upper bound.
Clearly this discrete NEP satisfies conditions (i) and (ii) in Definition \ref{def: 2groups}. Condition (iii) in Definition \ref{def: 2groups} simply requires that two groups of products $G_1$ and $G_2$ exist such that: two products of firms $\nu$ and $\mu$ belonging to the same group are substitutes (that is $c^\nu_\mu \geq 0$ and $c^\mu_\nu \geq 0$), while two products of firms $\nu$ and $\mu$ belonging to different groups are complements (that is $c^\nu_\mu \leq 0$ and $c^\mu_\nu \leq 0$).

Let us consider the following Jacobi-type method:
\begin{algorithm}\label{alg: Jacobi} (Jacobi-type Method)
 \begin{description}
  \item[(S.0)] Choose a starting point $\xbold^0 \in X \cap \Z^n$ and set $k:=0$.
  \item[(S.1)] If $\xbold^k$ is a solution of the discrete NEP then STOP.
  \item[(S.2)] Choose a subset ${\cal J}^k$ of the players. For each $\nu \in {\cal J}^k$ compute a best response ${\hat x}^{k,\nu}$:
  $$
   {\hat x}^{k,\nu} \in \hat x^\nu(\xbold^{k,-\nu}).
  $$
  \item[(S.3)] For all $\nu \in {\cal J}^k$ set $x^{k+1,\nu} := \hat x^{k,\nu}$ and for all $\nu \notin {\cal J}^k$ set $x^{k+1,\nu} := x^{k,\nu}$. \\ Set $k:=k+1$ and go to \textbf{\emph{(S.1)}}.
 \end{description}
\end{algorithm}
\medskip

\noindent
We recall that $\hat x^\nu(\bullet)$ is defined in \eqref{eq: best response}.
This type of methods are very popular among practitioners of Nash problems and their rationale is particularly simple to grasp since they are the most ``natural'' decomposition methods. Algorithm \ref{alg: Jacobi} can be easily implemented in a parallel framework in order to reduce the computational burden at \textbf{(S.2)}. Furthermore, it has a non-standard feature that gives an additional degree of freedom compared to traditional Jacobi-type schemes: the choice of subset ${\cal J}^k$ of the players that ``play'' at iteration $k$. We can say that it is an incomplete Jacobi-type iteration. As special cases, by selecting only one player at each iteration we get a Gauss-Southwell scheme, while by selecting roundly each single player we get a Gauss-Seidel scheme.
\begin{theorem}\label{th: Jacobi convergence}
 Suppose that \eqref{eq: new Z condition 1} and \eqref{eq: new Z condition 2} hold and that, for each player $\nu$, $\theta_\nu$ is defined as in \eqref{eq: objective quadratic generalized}, where $\theta_\nu^P$ is defined as in \eqref{eq: private quadratic generalized} and $\Theta^O_\nu$ is an operator made up of convex or concave functions.
 Suppose that $X$ is non-empty, bounded and defined by box constraints as in \eqref{eq: box}.

 For all $\nu \in \{1,\ldots,N\}$ and all $i \in \{1,\ldots,n_\nu\}$,
 let $x^{0,\nu}_i := l^\nu_i$ if $(\nu,i) \in G_1$ and 
 let $x^{0,\nu}_i := u^\nu_i$ if $(\nu,i) \in G_2$.
 Let a finite positive integer $h$ exists such that $\nu \in \cup_{t = k}^{k + h} {\cal J}^t$ for each player $\nu$ and each iterate $k$.
 
 Then, for each iterate $k$ and each player $\nu \in {\cal J}^k$, a best response $\hat x^{k,\nu} \in \hat x^\nu(\xbold^{k,-\nu})$ can be computed such that
 \begin{equation}\label{eq: best increase}
 \hat x^{k,\nu}_i \geq x^{k,\nu}_i, \quad \forall \, (\nu,i) \in G_1, \qquad \hat x^{k,\nu}_i \leq x^{k,\nu}_i, \quad \forall \, (\nu,i) \in G_2.
 \end{equation}
 By computing $\hat x^{k,\nu}$ for all $k$ and all $\nu \in {\cal J}^k$ such that \eqref{eq: best increase} holds, Algorithm \ref{alg: Jacobi} converges in a finite number of iterations to a solution of the discrete NEP defined by \eqref{eq: prob}.
\end{theorem}
\begin{proof}
 First of all note that, by assumptions on $X$, set $\hat x^\nu(\bar \xbold^{-\nu})$ is non-empty and finite for any $\nu$ and any $\bar \xbold^{-\nu}$. However, it is easy to see that $\hat x^\nu(\bar \xbold^{-\nu})$ may contain more than one element.

 Let us consider the first iteration. Since $\xbold^1 \in X \cap \Z^n$ then it holds that $x^{1,\nu}_i \geq x^{0,\nu}_i$ if $(\nu,i) \in G_1$, and $x^{1,\nu}_i \leq x^{0,\nu}_i$ if, otherwise, $(\nu,i) \in G_2$.
 For all $\nu$ and all $i \in \{1,\ldots,n_\nu\}$ such that $(\nu,i) \in G_1$, if $\Theta_{\nu,i}^O$ is a convex function, then we can write
 \begin{equation}\label{eq: jacobi proof conv}
  \Theta_{\nu,i}^O (\xbold^{0,-\nu}) - \Theta_{\nu,i}^O (\xbold^{1,-\nu}) \geq \nabla \Theta_{\nu,i}^O (\xbold^{1,-\nu})\trt (\xbold^{0,-\nu} - {\xbold^{1,-\nu}}),
 \end{equation}
 otherwise, $\Theta_{\nu,i}^O$ is concave, and then we can write
 \begin{equation}\label{eq: jacobi proof conc}
  \Theta_{\nu,i}^O (\xbold^{1,-\nu}) - \Theta_{\nu,i}^O (\xbold^{0,-\nu}) \leq \nabla \Theta_{\nu,i}^O (\xbold^{0,-\nu})\trt (\xbold^{1,-\nu} - {\xbold^{0,-\nu}}).
 \end{equation}
 In both cases, since, by \eqref{eq: new Z condition 1}, for all $\mu \neq \nu$ such that $(\mu,j) \in G_1$ it holds that 
 $$
 x^{1,\mu}_j \geq x^{0,\mu}_j, \qquad \frac{\partial \Theta_{\nu,i}^O (\xbold^{0,-\nu})}{\partial x^\mu_j} \leq 0, \qquad \frac{\partial \Theta_{\nu,i}^O (\xbold^{1,-\nu})}{\partial x^\mu_j} \leq 0,
 $$
 while, by \eqref{eq: new Z condition 2}, for all $\mu \neq \nu$ such that $(\mu,j) \in G_2$ it holds that
 $$
 x^{1,\mu}_j \leq x^{0,\mu}_j, \qquad \frac{\partial \Theta_{\nu,i}^O (\xbold^{0,-\nu})}{\partial x^\mu_j} \geq 0, \qquad \frac{\partial \Theta_{\nu,i}^O (\xbold^{1,-\nu})}{\partial x^\mu_j} \geq 0,
 $$
 then, by \eqref{eq: jacobi proof conv} or \eqref{eq: jacobi proof conc}, it holds that
 \begin{equation}\label{eq: decreasing gradients}
  \Theta_{\nu,i}^O (\xbold^{1,-\nu}) \leq \Theta_{\nu,i}^O (\xbold^{0,-\nu}), \qquad \forall \, (\nu,i) \in G_1.
 \end{equation}
 On the other hand, for all $\nu$ and all $i \in \{1,\ldots,n_\nu\}$ such that $(\nu,i) \in G_2$, if $\Theta_{\nu,i}^O$ is a convex function, then we can write
 \begin{equation}\label{eq: jacobi proof conv2}
  \Theta_{\nu,i}^O (\xbold^{1,-\nu}) - \Theta_{\nu,i}^O (\xbold^{0,-\nu}) \geq \nabla \Theta_{\nu,i}^O (\xbold^{0,-\nu})\trt (\xbold^{1,-\nu} - {\xbold^{0,-\nu}}),
 \end{equation}
 otherwise, $\Theta_{\nu,i}^O$ is concave, and we can write
 \begin{equation}\label{eq: jacobi proof conc2}
  \Theta_{\nu,i}^O (\xbold^{0,-\nu}) - \Theta_{\nu,i}^O (\xbold^{1,-\nu}) \leq \nabla \Theta_{\nu,i}^O (\xbold^{1,-\nu})\trt (\xbold^{0,-\nu} - {\xbold^{1,-\nu}}).
 \end{equation}
 Then, by \eqref{eq: new Z condition 2}, for all $\mu \neq \nu$ such that $(\mu,j) \in G_1$ it holds that 
 $$
 x^{1,\mu}_j \geq x^{0,\mu}_j, \qquad \frac{\partial \Theta_{\nu,i}^O (\xbold^{0,-\nu})}{\partial x^\mu_j} \geq 0, \qquad \frac{\partial \Theta_{\nu,i}^O (\xbold^{1,-\nu})}{\partial x^\mu_j} \geq 0,
 $$
 while, by \eqref{eq: new Z condition 1}, for all $\mu \neq \nu$ such that $(\mu,j) \in G_2$ it holds that
 $$
 x^{1,\mu}_j \leq x^{0,\mu}_j, \qquad \frac{\partial \Theta_{\nu,i}^O (\xbold^{0,-\nu})}{\partial x^\mu_j} \leq 0, \qquad \frac{\partial \Theta_{\nu,i}^O (\xbold^{1,-\nu})}{\partial x^\mu_j} \leq 0,
 $$
 then, by \eqref{eq: jacobi proof conv2} or \eqref{eq: jacobi proof conc2}, it holds that
 \begin{equation}\label{eq: increasing gradients}
  \Theta_{\nu,i}^O (\xbold^{1,-\nu}) \geq \Theta_{\nu,i}^O (\xbold^{0,-\nu}), \qquad \forall \, (\nu,i) \in G_2.
 \end{equation}
 Now let us consider the second iteration.
 In order to prove that, for all $\nu$, a best response $\hat x^{1,\nu} \in \hat x^\nu(\xbold^{1,-\nu})$ exists such that for all $i \in \{1,\ldots,n_\nu\}$:
 \begin{equation}\label{eq: existence increase}
  \hat x^{1,\nu}_i \geq x^{1,\nu}_i, \quad \text{ if } \, (\nu,i) \in G_1, \qquad \text{ and, } \qquad \hat x^{1,\nu}_i \leq x^{1,\nu}_i, \quad \text{ if } \, (\nu,i) \in G_2,
 \end{equation}
 we have to consider two possibilities. If $\nu \notin {\cal J}^0$ then $x^{1,\nu} = x^{0,\nu}$ and then \eqref{eq: existence increase} is trivially satisfied. Otherwise $\nu \in {\cal J}^0$ and we suppose by contradiction that for all $y^\nu \in \hat x^\nu(\xbold^{1,-\nu})$ a non-empty set of indices $J \subseteq \{1, \ldots, n_\nu\}$ exists such that for all $i \in J$ it holds that $y^\nu_i < x^{1,\nu}_i$ if $(\nu,i) \in G_1$ and $y^\nu_i > x^{1,\nu}_i$ if $(\nu,i) \in G_2$. Now we show that this is impossible.
 Let $\bar J := \{1, \ldots, n_\nu\} \setminus J$, for all $j \in \bar J$ we have $y^\nu_j \geq x^{1,\nu}_j$ if $(\nu,j) \in G_1$ and $y^\nu_j \leq x^{1,\nu}_j$ if $(\nu,j) \in G_2$.
 We define $\bar y^\nu, \tilde y^\nu \in \Z^{n_\nu}$ such that $\bar y^\nu_J = y^\nu_J$, $\bar y^\nu_{\bar J} = x^{1,\nu}_{\bar J}$, $\tilde y^\nu_J = x^{1,\nu}_J$ and $\tilde y^\nu_{\bar J} = y^\nu_{\bar J}$.
 It is easy to see that both $\bar y^\nu$ and $\tilde y^\nu$ are feasible for player $\nu$, since $X$ is a box.
 In order to show that
 \begin{equation}\label{eq: dominance condition}
  \left( \theta_\nu^P (\tilde y^\nu) - \theta_\nu^P (y^\nu) \right) - \left( \theta_\nu^P (x^{1,\nu}) - \theta_\nu^P (\bar y^\nu) \right) \leq 0,
 \end{equation}
 we consider the following two chains of inequalities:
 \begin{align*}
        &\left[ \frac{1}{2} ({\tilde y^\nu})\trt Q^\nu \tilde y^\nu - \frac{1}{2} ({y^\nu})\trt Q^\nu y^\nu \right] - \left[ \frac{1}{2} (x^{1,\nu})\trt Q^\nu (x^{1,\nu}) - \frac{1}{2} ({\bar y^\nu})\trt Q^\nu \bar y^\nu \right] = \\
        &\left[ \left( (y^\nu_J)\trt \; (y^\nu_{\bar J})\trt \right) \, \left( \begin{array}{c} Q^\nu_{JJ} \\ Q^\nu_{\bar JJ} \end{array} \right) \, \left( x^{1,\nu}_J - y^\nu_J \right) + \frac{1}{2} \left( x^{1,\nu}_J - y^\nu_J \right)\trt Q^\nu_{JJ} \left( x^{1,\nu}_J - y^\nu_J \right) \right] - \\
        &\left[ \left( (y^\nu_J)\trt \; (x^{1,\nu}_{\bar J})\trt \right) \, \left( \begin{array}{c} Q^\nu_{JJ} \\ Q^\nu_{\bar JJ} \end{array} \right) \, \left( x^{1,\nu}_J - y^\nu_J \right) + \frac{1}{2} \left( x^{1,\nu}_J - y^\nu_J \right)\trt Q^\nu_{JJ} \left( x^{1,\nu}_J - y^\nu_J \right) \right] = \\
        &\left( y^\nu_{\bar J} - x^{1,\nu}_{\bar J} \right)\trt Q^\nu_{\bar JJ} \left( x^{1,\nu}_J - y^\nu_J \right) \leq 0,
       \end{align*}
  where the last inequality holds because:
  \begin{itemize}
  \item for all $j \in \bar J$: $\left( y^\nu_j - x^{1,\nu}_j \right) \geq 0$ if $(\nu,j) \in G_1$ and $\left( y^\nu_j - x^{1,\nu}_j \right) \leq 0$ if $(\nu,j) \in G_2$,
  \item for all $j \in \bar J$ and all $i \in J$: $Q^\nu_{ji} \leq 0$ if $(\nu,j) \in G_1 \ni (\nu,i)$ or $(\nu,j) \in G_2 \ni (\nu,i)$, by \eqref{eq: new Z condition 1}, and $Q^\nu_{ji} \geq 0$ if $(\nu,j) \in G_1 \not\ni (\nu,i)$ or $(\nu,j) \in G_2 \not\ni (\nu,i)$, by \eqref{eq: new Z condition 2},
  \item for all $i \in J$: $\left( x^{1,\nu}_i - y^\nu_i \right) > 0$ if $(\nu,i) \in G_1$ and $\left( x^{1,\nu}_i - y^\nu_i \right) < 0$ if $(\nu,i) \in G_2$;
  \end{itemize}
  the other chain is the following:
  \begin{align*}
        &\sum_{i=1}^{n_\nu} \vartheta_{\nu,i}^P (\tilde y^\nu_i) - \sum_{i=1}^{n_\nu} \vartheta_{\nu,i}^P (y^\nu_i) - \sum_{i=1}^{n_\nu} \vartheta_{\nu,i}^P (x^{1,\nu}_i) + \sum_{i=1}^{n_\nu} \vartheta_{\nu,i}^P (\bar y^\nu_i) = \\
        &\left(\sum_{i \in J} \vartheta_{\nu,i}^P (x^{1,\nu}_i) + \sum_{i \in \bar J} \vartheta_{\nu,i}^P (y^\nu_i)\right) - \sum_{i=1}^{n_\nu} \vartheta_{\nu,i}^P (y^\nu_i) - \sum_{i=1}^{n_\nu} \vartheta_{\nu,i}^P (x^{1,\nu}_i) + \\
        &\quad \left(\sum_{i \in J} \vartheta_{\nu,i}^P (y^\nu_i) + \sum_{i \in \bar J} \vartheta_{\nu,i}^P (x^{1,\nu}_i)\right) = 0;
       \end{align*}
 therefore \eqref{eq: dominance condition} holds. By using \eqref{eq: dominance condition}, we can write the following chain of inequalities
 \begin{align*}
  0 &\geq \theta_\nu^P (\tilde y^\nu) - \theta_\nu^P (y^\nu) - \theta_\nu^P (x^{1,\nu}) + \theta_\nu^P (\bar y^\nu) \\
    &\overset{(A)}{\geq} \theta_\nu^P (\tilde y^\nu) - \theta_\nu^P (y^\nu) + \Theta_{\nu,J}^O (\xbold^{0,-\nu})\trt (x^{1,\nu}_J - y^\nu_J) \\
    &\overset{(B)}{\geq} \theta_\nu^P (\tilde y^\nu) - \theta_\nu^P (y^\nu) + \Theta_{\nu,J}^O (\xbold^{1,-\nu})\trt (x^{1,\nu}_J - y^\nu_J) \\
    &= \theta_\nu^P (\tilde y^\nu) - \theta_\nu^P (y^\nu) + \Theta_\nu^O (\xbold^{1,-\nu})\trt (\tilde y^\nu - y^\nu),
 \end{align*}
 where (A) holds since $x^{1,\nu} \in \hat x^\nu(\xbold^{0,-\nu})$ (remember that we are considering the case in which $\nu \in {\cal J}^0$) and $\bar y^\nu$ is feasible for player $\nu$; while (B) is true since for all $i \in J$: if $(\nu,i) \in G_1$ we have \eqref{eq: decreasing gradients} and $(x^{1,\nu}_i - y^\nu_i) > 0$, and if $(\nu,i) \in G_2$ we have \eqref{eq: increasing gradients} and $(x^{1,\nu}_i - y^\nu_i) < 0$.
 Then we can conclude that $\theta_\nu (\tilde y^\nu, \xbold^{1,-\nu}) \leq \theta_\nu (y^\nu, \xbold^{1,-\nu})$ and, since $\tilde y^\nu_i \geq x^{1,\nu}_i$ for all $(\nu,i) \in G_1$ and $\tilde y^\nu_i \leq x^{1,\nu}_i$ for all $(\nu,i) \in G_2$, this is a contradiction.
 Therefore, for all $\nu \in {\cal J}^1$, we can set $\hat x^{1,\nu} \in \hat x^\nu(\xbold^{1,-\nu})$ satisfying \eqref{eq: existence increase} for all $i \in \{1,\ldots,n_\nu\}$, and then we obtain  \begin{equation*}\label{eq: existence increase x2}
  x^{2,\nu}_i \geq x^{1,\nu}_i, \quad \forall \, (\nu,i) \in G_1, \qquad x^{2,\nu}_i \leq x^{1,\nu}_i, \quad \forall \, (\nu,i) \in G_2.
 \end{equation*}
 At a generic iterate $k \geq 2$, assuming that for all $t<k$
 \begin{equation*}\label{eq: existence increase xt}
  x^{k,\nu}_i \geq x^{t,\nu}_i, \quad \forall \, (\nu,i) \in G_1, \qquad x^{k,\nu}_i \leq x^{t,\nu}_i, \quad \forall \, (\nu,i) \in G_2,
 \end{equation*}
 we can do similar considerations in order to prove that we can get for all $(\nu,i)$:
  \begin{equation}\label{eq: existence increase xk}
  x^{k+1,\nu}_i \geq x^{k,\nu}_i, \quad \text{ if } \, (\nu,i) \in G_1, \qquad \text{ and, } \qquad x^{k+1,\nu}_i \leq x^{k,\nu}_i, \quad \text{ if } \, (\nu,i) \in G_2.
 \end{equation}
 In fact, we have the following two possibilities for any $\nu \in {\cal J}^k$: if $\nu \notin \cup_{t=0}^{k-1} {\cal J}^t$ then $x^{k,\nu} = x^{0,\nu}$ and then \eqref{eq: existence increase xk} is trivially satisfied, otherwise,
 as above we can contradict the fact that any point $y^\nu$ in $\hat x^\nu(\xbold^{k,-\nu})$ has a non-empty set of indices $J$ such that for all $i \in J$ it holds that $y^\nu_i < x^{k,\nu}_i$ if $(\nu,i) \in G_1$ and $y^\nu_i > x^{k,\nu}_i$ if $(\nu,i) \in G_2$. We define $\bar y^\nu$ and $\tilde y^\nu$ as above. Let $p<k$ be the last iterate such that $\nu \in {\cal J}^p$. By the same considerations made above, we can write $\Theta_{\nu,i}^O (\xbold^{k,-\nu}) \leq \Theta_{\nu,i}^O (\xbold^{p,-\nu})$ for all $(\nu,i) \in G_1$, $\Theta_{\nu,i}^O (\xbold^{k,-\nu}) \geq \Theta_{\nu,i}^O (\xbold^{p,-\nu})$ for all $(\nu,i) \in G_2$, and $\left( \theta_\nu^P (\tilde y^\nu) - \theta_\nu^P (y^\nu) \right) - \left( \theta_\nu^P (x^{k,\nu}) - \theta_\nu^P (\bar y^\nu) \right) \leq 0$ and therefore, since $x^{k,\nu} = x^{p+1,\nu}$, the chain of inequalities
 \begin{align*}
  0 &\geq \theta_\nu^P (\tilde y^\nu) - \theta_\nu^P (y^\nu) - \theta_\nu^P (x^{k,\nu}) + \theta_\nu^P (\bar y^\nu) \\
    &\geq \theta_\nu^P (\tilde y^\nu) - \theta_\nu^P (y^\nu) + \Theta_{\nu,J}^O (\xbold^{p,-\nu})\trt (x^{k,\nu}_J - y^\nu_J) \\
    &\geq \theta_\nu^P (\tilde y^\nu) - \theta_\nu^P (y^\nu) + \Theta_{\nu,J}^O (\xbold^{k,-\nu})\trt (x^{k,\nu}_J - y^\nu_J) \\
    &= \theta_\nu^P (\tilde y^\nu) - \theta_\nu^P (y^\nu) + \Theta_\nu^O (\xbold^{k,-\nu})\trt (\tilde y^\nu - y^\nu),
 \end{align*}
 which proves the contradiction in the same way as above.

 Therefore, we can say that the entire sequence $\{\xbold^k\}$ is such that $\xbold^k \in X \cap \Z^n$ and \eqref{eq: existence increase xk} is true for all $(\nu,i)$. Finally, by recalling that $\nu \in \cup_{t = k}^{k + h} {\cal J}^t$ for each player $\nu$ and each iterate $k$ and that $X$ is convex and compact, and therefore the set $X \cap \Z^n$ has a finite number of elements, we can conclude that the sequence $\{\xbold^k\}$ converges in a finite number of iterations to a point $\xbold^*$ such that
 $$
  x^{*,\nu} \in \hat x^\nu(\xbold^{*,{-\nu}}), \qquad \forall \, \nu \in \{1, \ldots, N\}.
 $$
 And, therefore, $\xbold^*$ is a solution for the discrete NEP.
\end{proof}
The following result is about complexity of Algorithm \ref{alg: Jacobi}.
\begin{proposition}\label{pr: complexity Jacobi}
 Let us suppose that all assumptions in Theorem \ref{th: Jacobi convergence} are fulfilled. Then Algorithm \ref{alg: Jacobi} converges to a solution of the discrete NEP defined by \eqref{eq: prob} in at most $h \left[ \sum_{\nu=1}^N \sum_{i=1}^{n_\nu} \left(u^\nu_i-l^\nu_i\right)\right] + h$ steps.
\end{proposition}
\begin{proof}
 As stated in the proof of Theorem \ref{th: Jacobi convergence}, sequence $\{\xbold^k\}$, generated by Algorithm \ref{alg: Jacobi}, is such that $\xbold^k \in X \cap \Z^n$ and \eqref{eq: existence increase xk} is true for all $(\nu,i)$. By recalling that $\nu \in \cup_{t = k}^{k + h} {\cal J}^t$ for each player $\nu$ and each iterate $k$, then, if for an iteration $\bar k$ it holds that $\xbold^{\bar k} = \xbold^{\bar k+h+1}$, then $\xbold^{\bar k+h+1}$ is a solution of the discrete NEP. Otherwise at least one $(\nu,i)$ exists such that
 \begin{equation*}
  x^{\tilde k+1,\nu}_i > x^{\tilde k,\nu}_i, \quad \text{ if } \, (\nu,i) \in G_1, \qquad \text{ or, } \qquad x^{\tilde k+1,\nu}_i < x^{\tilde k,\nu}_i, \quad \text{ if } \, (\nu,i) \in G_2,
 \end{equation*}
 where $\bar k \leq \tilde k \leq \bar k + h$.
 Therefore Algorithm \ref{alg: Jacobi} can do no more than $h \sum_{\nu=1}^N \sum_{i=1}^{n_\nu} \left(u^\nu_i-l^\nu_i\right)$ steps before
 $x^{k,\nu}_i = u^\nu_i$ if $(\nu,i) \in G_1$ and 
 $x^{k,\nu}_i = l^\nu_i$ if $(\nu,i) \in G_2$, for all $\nu \in \{1,\ldots,N\}$ and all $i \in \{1,\ldots,n_\nu\}$,
 and no more than $h$ final steps to check the optimality.
\end{proof}
Note that, if the iterations of Algorithm \ref{alg: Jacobi} follow a Gauss-Seidel rule, then the upper bound for the algorithm steps is $N [ \sum_{\nu=1}^N \sum_{i=1}^{n_\nu} \left(u^\nu_i-l^\nu_i\right) ] + N$. While, if the iterations follow a Jacobi rule, the bound is $[\sum_{\nu=1}^N \sum_{i=1}^{n_\nu} \left(u^\nu_i-l^\nu_i\right)] + 1$.

The following corollary gives conditions for existence of equilibria of \emph{2-groups partitionable} discrete NEPs.
\begin{corollary}\label{co: existence}
 A \emph{2-groups partitionable} discrete NEP with $X$ non-empty and bounded always has at least one equilibrium.
\end{corollary}
Let us compare the class of discrete NEPs defined in this section with supermodular games defined in \cite{topkis1998supermodularity}. It is not difficult to see that the problem in Example \ref{ex: class easy} is not a supermodular game. In that, $-\theta_1$ is not supermodular in $x^1$ on $X^1$ for all $x^2$ in $X^2$, e.g.:
\begin{align*}
&-\theta_1 \left( (1,0)\trt, (0,0)\trt \right)
-\theta_1 \left( (0,1)\trt, (0,0)\trt \right)
= -12 > \\
&\qquad -13 = 
-\theta_1 \left( (0,0)\trt, (0,0)\trt \right)
-\theta_1 \left( (1,1)\trt, (0,0)\trt \right).
\end{align*}
Then the discrete NEP in Example \ref{ex: class easy} is a simple case for which existence conditions of Corollary \ref{co: existence} are satisfied, while those in \cite{topkis1998supermodularity} are not.
Moreover, it is not difficult to see that supermodularity conditions given in \cite{topkis1998supermodularity} are satisfied, in the framework we are considering, only if $JF(\xbold)$ is a $Z$-matrix for all $\xbold \in X$. We recall that a square matrix $M$ is a $Z$-matrix if all its off-diagonal entries are non-positive, see \cite{CottPangStoBk}. Clearly, any $JF(\xbold)$ which is a $Z$-matrix for all $\xbold \in X$ satisfies \eqref{eq: new Z condition 1} and \eqref{eq: new Z condition 2}, simply by putting all variables in the same group of the partition, $G_1$ or $G_2$. Therefore, we can say that the class of discrete NEPs defined here is more general than that defined by using supermodularity theory.

Let us consider again the pricing game described above. When there are only two firms ($N=2$), it is known as the Bertrand model of duopoly, see e.g. \cite{tirole1988theory}. In \cite{yang2008solutions} existence for this 2-players discrete NEP was proved only when the two products are substitutes, while, by putting one product in $G_1$ and the other in $G_2$ and referring to Corollary \ref{co: existence}, we can prove existence also when the two products are complements.

\section{Numerical experiments}\label{sec: numerical}

We tested Algorithms \ref{alg: class} and \ref{alg: Jacobi} on a benchmark of discrete NEPs in which $\theta_\nu$ is defined as in \eqref{eq: quadratic functions and box constraints}, for all $\nu$, and X is defined by box constraints as in \eqref{eq: box}.

All experiments were carried out on an Intel Core i7-4702MQ CPU @ 2.20GHz x 8 with Ubuntu 14.04 LTS 64-bit and by using Matlab 7.14.0.739 (R2012a).

In all our implementations, we computed a point in $\hat x^\nu(\xbold^{-\nu})$, defined in \eqref{eq: best response}, by using \verb ga , which is the genetic algorithm available in the considered Matlab release. Such algorithm is an easy-to-use procedure that proved to be effective in all our tests.

For what concerns Algorithm \ref{alg: branching}:
 the strategy for picking up from and inserting in $\cal L$ was First In First Out;
 $\cal O$ was implemented by using \verb ga  with standard options;
 $\cal S$ was implemented by using a C version of the PATH solver, see \cite{dirkse1995path}, with a Matlab interface downloaded from \verb http://pages.cs.wisc.edu/~ferris/path/  and whose detailed description can be found in \cite{ferris1999interfaces}; entries of points, returned by $\cal S$, that were less than 1e-10 far from their nearest integer value were rounded to their nearest integer value; $\cal F$ was implemented as in Algorithm \ref{alg: procedure F}; $\cal C$ was implemented as in Algorithm \ref{alg: procedure C}.

Step (S.4) in Algorithms \ref{alg: lower bound} and \ref{alg: upper bound} was computed by using a simple line search.

About Algorithm \ref{alg: Jacobi}: we used Gauss-Seidel iterations; the starting guess was set by following the assumptions in Theorem \ref{th: Jacobi convergence} and by using the first row of $JF$ to define $G_1$ and $G_2$; we used $\xbold^k=\xbold^{k-1}$ as stopping criterion at (S.1); the best responses were computed, at (S.2), by using \verb ga  with standard options.

To create a benchmark for testing Algorithm \ref{alg: class},
we randomly generated problems with $JF$ positive definite and non symmetric by keeping to the following procedure.
Let $M \in \Re^{n\times n}$ be a symmetric positive definite matrix, let $h > 0$ and let $M^{\max}$ be the maximum among all entries modules $|M_{ij}|$. We made $M$ to be non symmetric for all off diagonal entries $M_{ij}$ and $M_{ji}$, corresponding to blocks $C$ in \eqref{eq: quadratic functions and box constraints}, by doing the following operations: (1) compute a random value $v_{ij}$ uniformly in $[-h M^{\max}, h M^{\max}]$, (2) set $M_{ij} := M_{ij} + v_{ij}$ and $M_{ji} := M_{ji} - v_{ij}$. It is clear that, after this transformation, $M$ is still positive definite. In our tests, we considered two values of $h$, one higher ($h = 0.1$) and one lower ($h = 0.01$), in order to obtain matrices having different degrees of asymmetry.
In Table \ref{tab: problem list} we report the list of discrete NEPs used to test Algorithm \ref{alg: class}.
\begin{table}[!ht]
\centering
 \begin{tabular}{l|cccccc|c}
  \hline
  {\bf prob} & $\mathbf{\lambda_{\min}}$ & $\mathbf{\lambda_{\max}}$ & $\mathbf{b_{\inf}}$ & $\mathbf{b_{\sup}}$ & $\mathbf{l}$ & $\mathbf{u}$ & $\mathbf{\# feas}$ \\\hline\hline
  G-2-1-A-$\star$ & 0.03 & 0.25 & -1 & 1 & -500 & 500 & 1.002e+06
 \\\hline
  G-3-1-A-$\star$ & 0.04 & 1.29 & -1 & 1 & -50 & 50 & 1.030e+06
 \\\hline
  G-4-1-A-$\star$, G-2-2-A-$\star$ & 0.18 & 14.05 & -1 & 1 & -5 & 5 & 1.464e+04
 \\\hline
  G-6-1-A-$\star$, G-3-2-A-$\star$, G-2-3-A-$\star$ & 0.15 & 5.44 & -1 & 0 & 0 & 5 & 4.666e+04 \\\hline\hline
  G-2-1-B-$\star$ & 0.01 & 4.47 & -10 & 10 & -1000 & 1000 & 4.004e+06
 \\\hline
  G-3-1-B-$\star$ & 0.02 & 14.41 & -10 & 10 & -100 & 100 & 8.120e+06
 \\\hline
  G-4-1-B-$\star$, G-2-2-B-$\star$ & 0.41 & 18.14 & -10 & 10 & -10 & 10 & 1.945e+05 \\\hline
  G-6-1-B-$\star$, G-3-2-B-$\star$, G-2-3-B-$\star$ & 0.78 & 48.26 & -10 & 0 & 0 & 10 & 1.771e+06 \\\hline
 \end{tabular}
 \caption{Small test problems description. \label{tab: problem list}}
\end{table}
For each problem, defined in \eqref{eq: quadratic functions and box constraints} and \eqref{eq: box}, we indicate the following data:
{\bf prob} contains the problem name; $\mathbf{\lambda_{\min}}$ and $\mathbf{\lambda_{\max}}$ are the minimum and the maximum eigenvalues of the symmetric part of $JF$; $\mathbf{b_{\inf}}$ and $\mathbf{b_{\sup}}$ are the bounds with which all entries of $b^\nu \in \Re^{n_\nu}$ were generated by using a uniform distribution; $\mathbf{l}$ and $\mathbf{u}$ are the bounds of $X$; $\mathbf{\# feas}$ is the amount of feasible points.
Each name label tells us information about the type and dimensions of the problem. In particular, let $T$-$N$-$n$-$I$-$\star$ be a name label: $T$ indicates the problem type (problems in Table \ref{tab: problem list} are all of type ``G'', that is they are generic problems); $N$ is the amount of players; $n$ is the amount of variables for each player; $I$ indicates the specific instance; $\star$ indicates the value of $h$ (that is the degree of asymmetry of $JF$) and can be equal to ``H'' for $h=0.1$ or to ``L'' for $h=0.01$.

\begin{table}[!ht]
\centering
 \begin{tabular}{l|ccccccccc}
  \hline
 {\bf prob} & {\bf \#eq} & {\bf \%$_{1st}$} & {\bf \%$_{last}$} & {\bf \%$_{tot}$} & {\bf \%$_{LB}$} & {\bf \%$_F$} & {\bf O$_{1st}$} & {\bf O$_{last}$} & {\bf O$_{tot}$} \\\hline\hline
 ex \ref{ex: one} & 4 & 2.00 & 9.00 & 10.00 & 84.00 & 6.00 & 2 & 9 & 10 \\\hline
 ex \ref{ex: two} & 0 &  &  & 26.00 & 0.00 & 74.00 &  &  & 26 \\\hline
 ex \ref{ex: trap} & 3 & 6.25 & 50.00 & 56.25 & 0.00 & 43.75 & 1 & 8 & 9 \\\hline
 ex \ref{ex: class easy} & 2 & 0.01 & 0.01 & 14.22 & 9.09 & 76.69 & 1 & 2 & 2082 \\\hline\hline
 G-2-1-A-L & 3 & $<$ 0.01 & $<$ 0.01 & $<$ 0.01 & 99.99 & $<$ 0.01 & 2 & 7 & 7 \\\hline
 G-2-1-A-H & 2 & $<$ 0.01 & $<$ 0.01 & $<$ 0.01 & 99.99 & $<$ 0.01 & 2 & 3 & 4 \\\hline
 G-3-1-A-L & 8 & $<$ 0.01 & 0.01 & 0.01 & 99.93 & 0.06 & 1 & 141 & 142 \\\hline
 G-3-1-A-H & 5 & $<$ 0.01 & 0.01 & 0.01 & 99.97 & 0.02 & 4 & 78 & 108 \\\hline
 G-4-1-A-L & 11 & 0.01 & 4.95 & 8.73 & 36.36 & 54.91 & 2 & 725 & 1278 \\\hline
 G-4-1-A-H & 1 & 0.01 & 0.01 & 8.20 & 36.36 & 55.43 & 2 & 2 & 1201 \\\hline
 G-2-2-A-L & 3 & 0.03 & 1.76 & 8.68 & 36.36 & 54.96 & 4 & 258 & 1271 \\\hline
 G-2-2-A-H & 1 & 0.01 & 0.01 & 8.40 & 36.36 & 55.24 & 2 & 2 & 1230 \\\hline
 G-6-1-A-L & 5 & 0.01 & 0.17 & 3.60 & 88.89 & 7.51 & 5 & 78 & 1679 \\\hline
 G-6-1-A-H & 4 & $<$ 0.01 & 0.26 & 3.44 & 90.74 & 5.82 & 1 & 122 & 1605 \\\hline
 G-3-2-A-L & 4 & 0.01 & 0.17 & 3.59 & 88.89 & 7.52 & 5 & 78 & 1676 \\\hline
 G-3-2-A-H & 4 & 0.02 & 0.20 & 2.79 & 92.28 & 4.93 & 7 & 93 & 1301 \\\hline
 G-2-3-A-L & 3 & 0.01 & 0.17 & 3.59 & 88.89 & 7.52 & 5 & 78 & 1676 \\\hline
 G-2-3-A-H & 3 & 0.02 & 0.21 & 3.67 & 88.89 & 7.44 & 9 & 100 & 1713 \\\hline\hline
 G-2-1-B-L & 119 & $<$ 0.01 & 0.01 & 0.01 & 98.69 & 1.30 & 1 & 449 & 449 \\\hline
 G-2-1-B-H & 10 & $<$ 0.01 & $<$ 0.01 & $<$ 0.01 & 99.99 & $<$ 0.01 & 2 & 30 & 30 \\\hline
 G-3-1-B-L & 50 & $<$ 0.01 & 0.13 & 0.13 & 94.39 & 5.49 & 15 & 10185 & 10206 \\\hline
 G-3-1-B-H & 3 & $<$ 0.01 & $<$ 0.01 & $<$ 0.01 & 99.99 & $<$ 0.01 & 4 & 58 & 113 \\\hline
 G-4-1-B-L & 8 & 0.01 & 0.43 & 2.94 & 88.34 & 8.72 & 19 & 842 & 5714 \\\hline
 G-4-1-B-H & 0 &  &  & 2.40 & 66.83 & 30.77 &  &  & 4664 \\\hline
 G-2-2-B-L & 7 & $<$ 0.01 & 0.30 & 3.57 & 84.56 & 11.87 & 4 & 582 & 6940 \\\hline
 G-2-2-B-H & 1 & 0.04 & 0.04 & 4.34 & 79.59 & 16.06 & 69 & 69 & 8447 \\\hline
 G-6-1-B-L & 8 & $<$ 0.01 & 0.02 & 1.13 & 84.97 & 13.90 & 1 & 410 & 20015 \\\hline
 G-6-1-B-H & 8 & $<$ 0.01 & 0.07 & 0.59 & 93.69 & 5.72 & 3 & 1181 & 10446 \\\hline
 G-3-2-B-L & 7 & $<$ 0.01 & 0.06 & 1.14 & 84.97 & 13.89 & 1 & 979 & 20217 \\\hline
 G-3-2-B-H & 5 & $<$ 0.01 & 0.07 & 0.42 & 93.69 & 5.89 & 6 & 1237 & 7432 \\\hline
 G-2-3-B-L & 7 & $<$ 0.01 & 0.06 & 0.98 & 87.98 & 11.04 & 1 & 992 & 17366 \\\hline
 G-2-3-B-H & 4 & $<$ 0.01 & 0.01 & 0.44 & 92.79 & 6.77 & 6 & 121 & 7773 \\\hline
 \end{tabular}
 \caption{Results for Algorithm \ref{alg: class} (1/2). \label{tab: results 1 gen}}
\end{table}

\begin{table}[!ht]
\centering
 \begin{tabular}{l|cccccc}
  \hline
 {\bf prob} & {\bf iter} & {\bf t$_{LB}$} & {\bf t$_{1st}$} & {\bf t$_{last}$} & {\bf t$_{tot}$} & {\bf ga} \\\hline\hline
 ex \ref{ex: one} & 15 & $<$ 0.01 & 2.05 & 4.60 & 4.86 & 14 \\\hline
 ex \ref{ex: two} & 27 & $<$ 0.01 &  &  & 7.22 & 28 \\\hline
 ex \ref{ex: trap} & 13 & $<$ 0.01 & 0.51 & 3.07 & 3.33 & 13 \\\hline
 ex \ref{ex: class easy} & 2338 & $<$ 0.01 & 0.65 & 1.15 & 563.70 & 2117 \\\hline\hline
 G-2-1-A-L & 11 & $<$ 0.01 & 0.82 & 2.67 & 2.67 & 10 \\\hline
 G-2-1-A-H & 7 & $<$ 0.01 & 0.80 & 1.31 & 1.56 & 6 \\\hline
 G-3-1-A-L & 191 & $<$ 0.01 & 0.87 & 50.28 & 50.55 & 186 \\\hline
 G-3-1-A-H & 145 & $<$ 0.01 & 1.93 & 27.18 & 35.93 & 131 \\\hline
 G-4-1-A-L & 1598 & $<$ 0.01 & 2.00 & 280.12 & 467.00 & 1682 \\\hline
 G-4-1-A-H & 1514 & $<$ 0.01 & 1.69 & 1.69 & 428.63 & 1545 \\\hline
 G-2-2-A-L & 1588 & $<$ 0.01 & 1.41 & 72.77 & 344.31 & 1306 \\\hline
 G-2-2-A-H & 1514 & $<$ 0.01 & 0.87 & 0.87 & 333.00 & 1263 \\\hline
 G-6-1-A-L & 1989 & $<$ 0.01 & 4.01 & 51.52 & 809.02 & 2790 \\\hline
 G-6-1-A-H & 1823 & $<$ 0.01 & 1.80 & 76.37 & 757.40 & 2616 \\\hline
 G-3-2-A-L & 1983 & $<$ 0.01 & 2.31 & 30.35 & 510.24 & 1900 \\\hline
 G-3-2-A-H & 1494 & $<$ 0.01 & 2.69 & 32.55 & 388.70 & 1450 \\\hline
 G-2-3-A-L & 1980 & $<$ 0.01 & 2.06 & 25.01 & 465.69 & 1712 \\\hline
 G-2-3-A-H & 1966 & $<$ 0.01 & 3.36 & 30.31 & 473.35 & 1742 \\\hline\hline
 G-2-1-B-L & 687 & $<$ 0.01 & 0.56 & 191.82 & 191.83 & 752 \\\hline
 G-2-1-B-H & 51 & $<$ 0.01 & 0.78 & 11.70 & 11.70 & 46 \\\hline
 G-3-1-B-L & 11605 & $<$ 0.01 & 5.06 & 2785.80 & 2791.34 & 10532 \\\hline
 G-3-1-B-H & 155 & $<$ 0.01 & 1.60 & 17.84 & 33.57 & 128 \\\hline
 G-4-1-B-L & 6357 & $<$ 0.01 & 7.65 & 277.82 & 1753.61 & 6473 \\\hline
 G-4-1-B-H & 5289 & $<$ 0.01 &  &  & 1363.76 & 4963 \\\hline
 G-2-2-B-L & 7652 & $<$ 0.01 & 2.71 & 174.96 & 1948.00 & 7017 \\\hline
 G-2-2-B-H & 9012 & $<$ 0.01 & 19.89 & 19.89 & 2325.06 & 8500 \\\hline
 G-6-1-B-L & 22959 & $<$ 0.01 & 1.94 & 182.04 & 6536.53 & 22316 \\\hline
 G-6-1-B-H & 11662 & $<$ 0.01 & 2.45 & 394.01 & 3325.10 & 11708 \\\hline
 G-3-2-B-L & 23163 & $<$ 0.01 & 1.38 & 305.73 & 5585.99 & 20555 \\\hline
 G-3-2-B-H & 8395 & $<$ 0.01 & 2.28 & 349.63 & 2040.76 & 7575 \\\hline
 G-2-3-B-L & 20028 & $<$ 0.01 & 0.68 & 285.73 & 4756.06 & 17501 \\\hline
 G-2-3-B-H & 8865 & $<$ 0.01 & 2.00 & 34.45 & 2127.20 & 7829 \\\hline
 \end{tabular}
 \caption{Results for Algorithm \ref{alg: class} (2/2). \label{tab: results 2 gen}}
\end{table}

In Tables \ref{tab: results 1 gen} and \ref{tab: results 2 gen} we report the numerical results obtained by tackling all examples given in the paper and all problems in Table \ref{tab: problem list} with Algorithm \ref{alg: class}. More precisely: {\bf \#eq} is the amount of computed equilibria; {\bf \%$_{1st}$}, {\bf \%$_{last}$} and {\bf \%$_{tot}$} is the percentage of feasible points that were analyzed by Oracle ${\cal O}$ before the first equilibrium was found, the last equilibrium was found and list ${\cal L}$ was empty respectively; {\bf \%$_{LB}$} and {\bf \%$_F$} is the percentage of feasible points that was cut off by Algorithms \ref{alg: lower bound} and \ref{alg: upper bound} together and by Procedure ${\cal F}$ respectively; {\bf O$_{1st}$}, {\bf O$_{last}$} and {\bf O$_{tot}$} is the amount of feasible points that were analyzed by oracle ${\cal O}$ before the first equilibrium was computed, the last equilibrium was computed and list ${\cal L}$ was empty respectively; {\bf iter} is the total amount of iterations; {\bf t$_{LB}$}, {\bf t$_{1st}$}, {\bf t$_{last}$} and {\bf t$_{tot}$} is the elapsed CPU-time (in seconds) before Algorithms \ref{alg: lower bound} and \ref{alg: upper bound} stopped, the first equilibrium was computed, the last equilibrium was computed and the algorithm stopped respectively; {\bf ga} is the total amount of \verb ga  calls.

Results in Table \ref{tab: results 1 gen} show that almost always {\bf \%$_{last}$} is less than 0.50, that is the whole equilibrium set is computed by evaluating a tiny percentage of the feasible points. In any case, when {\bf \%$_{last}$} is bigger, {\bf O$_{last}$} is rather low. Most of the time {\bf \%$_{LB}$} plus {\bf \%$_F$} is over 95.00, this means that most feasible points were cut off by Algorithms \ref{alg: lower bound} and \ref{alg: upper bound} and by procedure ${\cal F}$. Therefore only a few percentage of the feasible points were evaluated by using the oracle.

Table \ref{tab: results 2 gen} shows that CPU-time used by Algorithms \ref{alg: lower bound} and \ref{alg: upper bound} is always negligible. We recall that procedure ${\cal S}$ was called one time per iteration. Therefore, by comparing {\bf iter} and {\bf ga} with {\bf t$_{tot}$}, we can easily deduce that, essentially, the whole elapsed time is consumed by \verb ga  calls. Moreover, note that, although {\bf O$_{tot}$} $\leq$ {\bf ga} $\leq$ $N${\bf O$_{tot}$}, most of the time {\bf ga} is just a bit bigger than {\bf O$_{tot}$}.

We tested Algorithm \ref{alg: Jacobi} on all examples given in the paper, on all problems in Table \ref{tab: problem list}, and on the big discrete NEPs described in Table \ref{tab: big problems}.
\begin{table}[!ht]
\centering
 \begin{tabular}{l|cccccc|c}
  \hline
  {\bf prob} & $\mathbf{\lambda_{\min}}$ & $\mathbf{\lambda_{\max}}$ & $\mathbf{b_{\inf}}$ & $\mathbf{b_{\sup}}$ & $\mathbf{l}$ & $\mathbf{u}$ & $\mathbf{\# feas}$ \\\hline\hline
  G-10-2-A-H & 0.01 & 2.33 & -10 & 10 & -5 & 5 & 6.727e+20
 \\\hline
  G-10-2-B-H & $<$ 0.01 & 2.75 & -10 & 10 & -5 & 5 & 6.727e+20
 \\\hline
  C-10-2 & 0.15 & 2.54 & -10 & 10 & -5 & 5 & 6.727e+20
 \\\hline
  C-8-10 & 0.01 & 2309.75 & -0.1 & 0.1 & -3 & 3 & 4.054e+67
 \\\hline
  C-20-5 & 0.10 & 362.08 & -10 & 10 & -10 & 10 & 1.667e+132
 \\\hline
  C-200-5 & 0.06 & 207.03 & -1000 & 1000 & -6 & 0 & $7^{1000}$
 \\\hline
 \end{tabular}
 \caption{Big test problems description. \label{tab: big problems}}
\end{table}
As above, the first character in the name label of the problems indicates the type of the problem. In Table \ref{tab: big problems}, some problems are of type ``C'', which indicates that these problems were randomly generated in order to belong to the class of \emph{2-groups partitionable} discrete NEPs (defined in Section \ref{sec:class}). Note that we are dealing with problems up to 1000 variables and $7^{1000}$ feasible points.

In Table \ref{tab: results class} we report the numerical results obtained by using Algorithm \ref{alg: Jacobi} to compute one equilibrium of the discrete NEPs considered. As above {\bf iter} is the total amount of iterations and {\bf ga} is the total amount of \verb ga  calls, while, {\bf t} is the elapsed CPU-time (in seconds) and {\bf G$_1$-G$_2$} has a checkmark if the problem is \emph{2-groups partitionable} (see Section \ref{sec:class}).
\begin{table}[!ht]
\centering
\begin{minipage}{0.49\textwidth}
\small
 \begin{tabular}{l|c|ccc}
  \hline
 {\bf prob} & {\bf G$_1$-G$_2$} & {\bf iter} & {\bf t} & {\bf ga} \\\hline\hline
 ex \ref{ex: one} & $\checkmark$ & 3 & 1.48 & 6 \\\hline
 ex \ref{ex: two} & & \multicolumn{3}{c}{failure} \\\hline
 ex \ref{ex: trap} & $\checkmark$ & 1 & 0.50 & 2 \\\hline
 ex \ref{ex: class easy} & $\checkmark$ & 2 & 1.13 & 4 \\\hline\hline
 G-2-1-A-L & $\checkmark$ & 15 & 7.42 & 30 \\\hline
 G-2-1-A-H & $\checkmark$ & 15 & 7.36 & 30 \\\hline
 G-3-1-A-L & $\checkmark$ & 16 & 12.29 & 48 \\\hline
 G-3-1-A-H & & 17 & 12.89 & 51 \\\hline
 G-4-1-A-L & & 5 & 5.21 & 20 \\\hline
 G-4-1-A-H & & 6 & 6.27 & 24 \\\hline
 G-2-2-A-L & & 4 & 2.47 & 8 \\\hline
 G-2-2-A-H & & \multicolumn{3}{c}{failure} \\\hline
 G-6-1-A-L & & 4 & 6.72 & 24 \\\hline
 G-6-1-A-H & & 5 & 8.62 & 30 \\\hline
 G-3-2-A-L & & 4 & 3.18 & 12 \\\hline
 G-3-2-A-H & & 5 & 3.83 & 15 \\\hline
 G-2-3-A-L & & 3 & 1.57 & 6 \\\hline
 G-2-3-A-H & & 3 & 1.58 & 6 \\\hline\hline
 G-2-1-B-L & $\checkmark$ & 250 & 123.26 & 500 \\\hline
\end{tabular}
\end{minipage}
\begin{minipage}{0.49\textwidth}
\hspace{-1em}
\small
 \begin{tabular}{l|c|ccc}
  \hline
 {\bf prob} & {\bf G$_1$-G$_2$} & {\bf iter} & {\bf t} & {\bf ga} \\\hline\hline
 G-2-1-B-H & $\checkmark$ & 43 & 21.37 & 86 \\\hline
 G-3-1-B-L & $\checkmark$ & 7 & 5.36 & 21 \\\hline
 G-3-1-B-H & & 12 & 9.43 & 36 \\\hline
 G-4-1-B-L & & 4 & 4.22 & 16 \\\hline
 G-4-1-B-H & & \multicolumn{3}{c}{failure} \\\hline
 G-2-2-B-L & & 4 & 1.99 & 8 \\\hline
 G-2-2-B-H & & 7 & 3.52 & 14 \\\hline
 G-6-1-B-L & & 4 & 6.75 & 24 \\\hline
 G-6-1-B-H & & 4 & 6.93 & 24 \\\hline
 G-3-2-B-L & & 4 & 3.08 & 12 \\\hline
 G-3-2-B-H & & 5 & 3.82 & 15 \\\hline
 G-2-3-B-L & & 5 & 2.62 & 10 \\\hline
 G-2-3-B-H & & 3 & 1.59 & 6 \\\hline\hline
 G-10-2-A-H & & 6 & 17.53 & 60 \\\hline
 G-10-2-B-H & & 5 & 14.61 & 50 \\\hline
 C-10-2 & $\checkmark$ & 5 & 14.62 & 50 \\\hline
 C-8-10 & $\checkmark$ & 2 & 11.10 & 16 \\\hline
 C-20-5 & $\checkmark$ & 6 & 54.63 & 120 \\\hline
 C-200-5 & $\checkmark$ & 9 & 2197.44 & 1800 \\\hline
\end{tabular}
\end{minipage}
 \caption{Results for Algorithm \ref{alg: Jacobi}. \label{tab: results class}}
\end{table}
Note that, although many problems do not satisfy conditions of Theorem \ref{th: Jacobi convergence}, nonetheless, Algorithm \ref{alg: Jacobi} always computed one equilibrium of the discrete NEP, except in three cases. In particular, it properly failed on ex \ref{ex: two} and G-4-1-B-H, since these discrete NEPs do not have any equilibrium, and it failed on G-2-2-A-H, that, anyway, does not satisfy assumptions of Theorem \ref{th: Jacobi convergence}. As concerns performances, Algorithm \ref{alg: Jacobi} compares with Algorithm \ref{alg: class} on small problems.
\begin{table}[!ht]
\centering
 \begin{tabular}{l|cccccc}
  \hline
 {\bf prob} & {\bf \%$_{1st}$} & {\bf O$_{1st}$} & {\bf iter} & {\bf t$_{LB}$} & {\bf t$_{1st}$} & {\bf ga} \\\hline\hline
 G-10-2-A-H & $<$ 0.01 & 38 & 126 & $<$ 0.01 & 47.58 & 158 \\\hline
 G-10-2-B-H & $<$ 0.01 & 167 & 337 & $<$ 0.01 & 117.51 & 395 \\\hline
 C-10-2 & $<$ 0.01 & 146 & 2192 & 0.02 & 138.17 & 467 \\\hline
 C-8-10 & \multicolumn{6}{c}{failure} \\\hline
 C-20-5 & \multicolumn{6}{c}{failure} \\\hline
 C-200-5 & \multicolumn{6}{c}{failure} \\\hline
 \end{tabular}
 \caption{Results for Algorithm \ref{alg: class} on big problems. \label{tab: results big}}
\end{table}
Table \ref{tab: results big} shows performances of Algorithm \ref{alg: class} when computing one equilirium of the big problems in Table \ref{tab: big problems}. There, we considered a failure a run that took more than 5000 seconds to stop.
Concluding, Algorithm \ref{alg: Jacobi} seems to be a big deal faster than the branching algorithm to compute one equilibrium of big discrete NEPs.

\section{Conclusions and directions for future research}

In this paper, we propose the first branching method to compute the whole solution set of any NEP with discrete strategy sets. This method works well on small and medium problems by efficiently computing all equilibria without examining more than a tiny percentage of the feasible points.
Futhermore a class of discrete games is defined for which we prove that a Jacobi-type algorithm converges to one of their equilibria. This algorithm is quite fast and works very well also on big problems.

Note that, although in this paper we do not tackle NEPs with mixed integer optimization problems, it is straightforward to extend our results to these more general games. In future work, we plan to address these mixed integer NEPs and to develop new methods for equilibrium selection in a mixed integer setting.

\bibliographystyle{plain}
\addcontentsline{toc}{chapter}{Bibliography}
\bibliography{Surbib}

\end{document}